\documentclass[oneside,english,11pt]{amsart}
\usepackage[T1]{fontenc}
\usepackage[utf8]{inputenc}
\usepackage{amsmath,amssymb,amsfonts,amsthm,amstext}
\usepackage{graphicx}
\usepackage{float}
\usepackage{tikz}
\usetikzlibrary{arrows.meta}
\definecolor{riem}{HTML}{E0EDFE}
\definecolor{lor}{HTML}{FEF2D8}
\usepackage[colorlinks=true,linkcolor=blue,
    citecolor=purple,      
    urlcolor=cyan,
]{hyperref}
\graphicspath{{../}}

\usepackage[parfill]{parskip}

\theoremstyle{plain}
\newtheorem{thm}{Theorem}[section]
\newtheorem{lem}[thm]{Lemma}
\newtheorem{prop}[thm]{Proposition}
\newtheorem{cor}[thm]{Corollary}
\newtheorem*{mainthm}{Theorem}
\newtheorem*{thmstar}{Theorem}
\newtheorem*{corstar}{Corollary}

\theoremstyle{definition}
\newtheorem{defn}[thm]{Definition}
\newtheorem{example}[thm]{Example}
\theoremstyle{remark}
\newtheorem{rem}[thm]{Remark}

\newcommand{\smallkeywords}[1]{\par\smallskip\noindent{\small\textbf{Keywords:} #1\par}}
\newcommand{\Rreg}{M_{R}}
\newcommand{\Lreg}{M_{L}}
\newcommand{\cMR}{\overline{M_{R}}}
\newcommand{\cML}{\overline{M_{L}}}
\newcommand{\HH}{\mathcal{H}}
\newcommand{\gt}{\widetilde{g}}
\newcommand{\Rad}{\operatorname{Rad}}

\newcommand{\intr}{\operatorname{int}}

\begin{document}

\title[Obstructions to Riemannian-Lorentzian Continuation]{Topological Obstructions to Riemannian-Lorentzian Continuation}

\author{Nathalie E. Rieger}
\address{Mathematics Department, Yale University, 219 Prospect Street, New Haven, CT 06520, USA}
\curraddr{Center for Cosmology and Particle Physics, Department of Physics, New York University, New York, NY 10003, USA}
\email{e.rieger@nyu.edu}

\author{\"{O}d\"{u}l Tetik}
\address{University of Vienna, Faculty of Physics, Mathematical Physics Group, Boltzmanngasse 5, 1090 Vienna, Austria}
\email{oeduel.tetik@univie.ac.at}

\date{}

\maketitle

\begin{abstract}
	Motivated by problems in instanton continuation, we establish geometro-topological selection rules for continuing a Riemannian region into a Lorentzian region through a smooth transverse change of signature. Given a collar-gluing $M=\overline{M_R} \cup_{\mathcal{H}}\overline{M_L}$ of a smooth manifold $M$ along a hypersurface $\mathcal{H}$, we show that an admissible signature-changing metric of the form $\widetilde g=g-fV^\flat\otimes V^\flat$ with $g$ a Riemannian metric, having Riemannian signature on $M_R$ and Lorentzian signature on $M_L$, exists if and only if a relative Euler class obstruction on $\overline{M_L}$ vanishes. This obstruction combines ordinary Euler characteristics $\chi$ in a manner that depends on which components of $\mathcal{H}$ represent big bangs and which represent big crunches. It also restricts the choice of $M_R$; for example, if $M$ is closed and $\dim M$ is even, then it requires$\chi(\overline{M_R})=\chi(M)$.
	
We also determine and classify the possible compactness types of $\overline{M_R}$, $\mathcal{H}$, and $\overline{M_L}$ occurring in such a collar-gluing construction, and construct explicit signature-changing metrics in each case. Topology change can be entirely contained within the Riemannian region when $M$ is the interior of a compact manifold with boundary.
\end{abstract}

\smallkeywords{signature change, singular semi-Riemannian metric, transverse type
change, relative Euler class, Morse theory, ends of manifolds, instanton, quantum cosmology.}

{
\setlength{\parskip}{0pt}
\tableofcontents
}

\section{Introduction}

The study of signature-changing manifolds is at the intersection of semi-Riemannian geometry, general relativity and quantum cosmology. An important motivation comes from the Hartle--Hawking no-boundary proposal~\cite{Hartle Hawking - Wave function of the Universe-1}, in which the big bang is replaced by an initial nonsingular Riemannian regime. Models of signature change also arise in classical relativity, loop quantum cosmology, and in Euclidean approaches to tunnelling~\cite{Ashtekar + Lewandowski - Background independent quantum gravity, Bojowald - Loop Quantum Cosmology, Cailleteau - Anomaly-freescalar perturbations, Ellis - Change of signature in classical relativity}. 

The question we address in this paper is global and topological:

\begin{quote}
Let $(M,g)$ be a Riemannian manifold and let $\HH\subset M$ be a hypersurface separating $M$ into two regions. When can $g$ be deformed into a transverse signature-changing metric with
transverse radical which is Riemannian on one side of $\HH$, Lorentzian on the other, and degenerate of corank one exactly along $\HH$?
\end{quote}

This perspective is inspired by instanton-continuation problems; see, for example, \cite{Kleban - Transitions}. Determining whether
a signature-changing manifold is an actual cosmological or gravitational instanton, however, requires additional analytic and dynamical input, such as field equations, boundary conditions, and regularity assumptions. These questions are not addressed in the present article. Instead, we develop a necessary geometric-topological foundation on which such questions can be posed. Our ansatz for deforming the Riemannian metric $g$ into a signature-changing metric of the desired type is
\begin{equation}\label{eq:ansatz}
\gt=g-fV^{\flat}\otimes V^{\flat}, \qquad f\in C^{\infty}(M),\quad V\in\Gamma(TM),
\end{equation}
where $f$ is a smooth function, $V$ a vector field, and $V^{\flat}=g(V,\cdot)$. This is the class of metrics studied in \cite{Hasse + Rieger-Transformation, Hasse + Rieger-Local Transformation, Hasse + Rieger-Loops}. It may be viewed as the signature-changing analogue of the construction $g-2V^\flat\otimes V^\flat$ recalled by~Hawking--Ellis~\cite[\S 2.6]{HawkingEllis1973}), which converts a Riemannian metric into a Lorentzian metric using a nowhere-vanishing vector field. The deformation is therefore carried out by two pieces of data: a scalar function $f$, which decides \emph{where} the signature changes, and a vector field $V$, which decides in which direction it changes. In particular, the orientation of \(V\) relative to a component of $\HH$ distinguishes a big-bang from a big-crunch transition, while on the Lorentzian side \(V\) defines the arrow of time. For every admissible continuation, \(V|_{M_L}\) is nowhere vanishing and timelike with respect to \(\widetilde g\). Hence the ansatz restricts attention to time-orientable Lorentzian regions, with \(V\) supplying a time orientation on each connected component of \(M_L\).

\subsection{Main results}

We call a decomposition 
\[
	M=\cMR\cup_{\HH}\cML
\]
a \emph{collar-gluing} if $\HH$ is a closed embedded hypersurface separating $M$, and if $\cMR,\,\cML$ are the closures of the two open regions in $M\setminus\HH$. We say that a collar-gluing is \emph{admissible} if there exist $f$ and $V$ such that the metric $\gt$ defined by~\eqref{eq:ansatz} is Riemannian on $\Rreg$, Lorentzian of index one on $\Lreg$, degenerate of corank one exactly along $\HH$, and a transverse type-changing metric with transverse radical along $\HH$, in the sense explained in Section~\ref{subsec:Preliminaries}. The Riemannian region together with its boundary, $\cMR$, is then the (potential) instanton, while $\Lreg$ is the Lorentzian universe.

We note that presenting the signature change by means of a collar-gluing, rather than beginning with a more abstract definition from which such a decomposition must later be recovered, is not an additional restriction. It follows directly from the transversality imposed on the signature change; see Remark~\ref{FORMIVA}.

Our main result formulates selection rules for a restricted class of signature-changing continuations. After identifying the obstruction to admissibility, we classify the possible global configurations according to the compactness properties of $\cMR$, $\cML$, and $\HH$. An additional topological obstruction appears when both $\cMR$ and $\cML$ are required to be noncompact while $\HH$ is compact: In that case, $M$ has at least two ends; see Section \ref{A227JE6}. The corresponding signature-changing metrics are then constructed in the main body of the paper.

\begin{mainthm}
\label{Theorem:Geometric Instanton Continuation}
Let \(M\) be a connected smooth manifold without boundary, with
\(\dim M=n\geq 2\), and let \(g\) be an auxiliary Riemannian metric on
\(M\). Suppose that \(M\) admits a collar-gluing decomposition
\[
    M=\overline{M_R}\cup_{\mathcal H}\overline{M_L},
\]
where $\partial\overline{M_R} = \partial\overline{M_L} = \mathcal H$ and $M_R:=\operatorname{int}(\overline{M_R})$, $M_L:=\operatorname{int}(\overline{M_L})$.
Then this collar-gluing determines an admissible Riemannian-Lorentzian continuation if and only if the Lorentzian side with boundary \(\overline{M_L}\) admits a smooth nowhere-vanishing vector field whose restriction to \(\mathcal H\) is transverse to \(\mathcal H\); equivalently, iff for some transverse boundary field \(\nu\), the relative Euler class
\[
    e(T\overline{M_L},\nu)
    \in
    H^n(\overline{M_L},\mathcal H;\mathbb Z_{w_1})
\]
vanishes. This obstruction vanishes on every
noncompact connected component of \(\overline{M_L}\), and on a compact connected component \(X\subset\overline{M_L}\) it is equivalent to the condition
\[
    \chi(X)=\chi(\partial_-X),
\]
where \(\partial_-X\subseteq\partial X\) is the union of boundary components along which \(\nu\) points into \(X\). In particular, if \(n\) is even, or if the transition is required to be uniformly big-bang-type or uniformly big-crunch-type, this reduces to $\chi(X)=0$ for every compact connected component \(X\subset\overline{M_L}\).

Moreover, the selection rules for the compactness types of $\cMR$, $\cML$ and $\HH$ are as follows, and each admissible configuration can be realized:
\begin{enumerate}
    \item[(A)] Compact \(\overline{M_R}\) and every component of
    \(\overline{M_L}\) noncompact: admissible.
    \item[(B)] Compact \(\overline{M_R}\) and compact \(\overline{M_L}\): admissible iff the relative Euler class of $\cML$ vanishes for some $\nu$ as above. 
    \item[(C)] Noncompact \(\overline{M_R}\), every component of \(\overline{M_L}\) noncompact, compact \(\mathcal H\): admissible iff \(M\) has at least two ends.
    \item[(D)] $\cMR$, $\HH$ both noncompact, every component of $\cML$ noncompact: admissible.
    \item[(E)] Noncompact \(\overline{M_R}\), compact
    \(\overline{M_L}\): admissible iff the relative Euler class of $\cML$ vanishes for some $\nu$ as above.
\end{enumerate}
In each case, $\cML$ can be allowed to have both compact and noncompact components provided that its compact components have vanishing relative Euler class for some $\nu$ as above.
\end{mainthm}

In each admissible case we show that there exist \(f\in C^\infty(M)\) and \(V\in\Gamma(TM)\), with \(V\) nowhere vanishing on \(M_L\cup\mathcal H\) and transverse to \(\mathcal H\), such that
\[
    \widetilde g=g-fV^\flat\otimes V^\flat
\]
is Riemannian on \(M_R\), Lorentzian of index one on \(M_L\), and degenerate of corank one precisely along \(\mathcal H\). Moreover, \(\widetilde g\) is a transverse type-changing metric with transverse radical along \(\mathcal H\).

\begin{rem}\label{rem:selection rules exhaustive}
	The conclusion and the selection rules are not dependent on Ansatz \ref{eq:ansatz} in that an arbitrary transverse Riemannian-Lorentzian type-changing metric with transverse radical (see Section \ref{subsec:Preliminaries}) corresponds to a collar-gluing of one of the five types listed in the theorem. Thus:
	\begin{quote}
		\emph{Equipping a smooth manifold with a transverse Riemannian-Lorentzian type-changing metric with transverse radical and time-orientable Lorentzian region is equivalent to presenting it as an admissible collar-gluing.}
	\end{quote}
\end{rem}

\begin{rem}The reason we formulate Cases (B) and (E) separately is that they are \emph{realised} differently.
\end{rem}

\begin{rem}\label{FORMIVA}
The collar-gluing hypothesis should not be viewed as an artificial additional assumption: Indeed, it records the hypersurface structure which is forced by transverse type change. If \(\widetilde g\) is a transverse type-changing metric and 
	\[
    \mathcal H
    =
    \{p\in M:\widetilde g_p\ \text{is degenerate, i.e.,}\ \det(\tilde g_p)=0\},
\]
the transversality condition
\[
    d\bigl(\det(\widetilde g_{\mu\nu})\bigr)\neq0
    \quad\text{on } \mathcal H
\]
implies that \(\mathcal H\) is a regular level set and thus a closed embedded hypersurface. Moreover, in the Riemannian-Lorentzian case considered here, the sign of \(\det(\widetilde g_{\mu\nu})\) distinguishes the Riemannian and Lorentzian regions; see Remark \ref{rem:twosided}. Thus \(\mathcal H\) is naturally two-sided and separating.
\end{rem}

\begin{rem}[the end-theoretic obstruction]
The obstruction in Case (C) is necessitated directly by the (non-)compactness requirements involved, so that the statement is really that there is no \emph{further} obstruction. For instance, $M\neq\mathbb{R}^n$ in Case (C) because $\mathbb{R}^n$ has only one end for $n\geq2$ (see Section \ref{A227JE6} and cf.\ Figure \ref{fig:compact-noncompact-H}).
\end{rem}

\begin{rem}[the obstruction is Lorentzian in origin but also affects the Riemannian region]\label{F1FECTV} As the theorem's statement makes clear, the obstruction originates entirely on the Lorentzian side (though it is not confined to it), because admissibility in our sense requires a nowhere-vanishing vector field on \(\overline{M_L}\) which is transverse to its boundary. This is the boundary analogue of the classical relation between Lorentzian metrics and nowhere-vanishing vector fields~\cite[Ch.~5, Prop.~37]{ONeill}. For a prescribed transverse boundary field, the obstruction to extending it to a nowhere-vanishing vector field on \(\overline{M_L}\) is the corresponding relative Euler class; see Remark~\ref{rem:euler}. On compact connected components this obstruction is expressed by the Euler-characteristic formula derived there, while on noncompact components it vanishes. In the metric produced by Ansatz~(1), this vector field is timelike on \(M_L\) and therefore supplies a time orientation on each connected component of \(M_L\).  There are two subtleties which are worked out in the main body of the paper:
\begin{itemize}
	\item When $\HH$ is disconnected, whether the boundary field points into $M_L$ or $M_R$ can change from component to component, and the relative obstruction takes into account (and is sensitive to) this decomposition of $\HH$. Figure \ref{fig:caps} illustrates this.
	\item The obstruction ``spills over'' into the Riemannian region through a Mayer--Vietoris identity relating $\chi(\overline{M_R})$, $\chi(\overline{M_L})$, $\chi(\mathcal H)$ and $\chi(M)$. See Corollary \ref{cor:MV}.
\end{itemize}
\end{rem}

\begin{rem}[C vs.\ D]
Cases (C) and (D) differ only in whether the transition hypersurface is compact (again,  cf.\ Figure \ref{fig:compact-noncompact-H}), which, in terms of our Ansatz is precisely the difference between a proper and a non-proper defining function. This makes our construction which realizes Case (D) quite different from the other cases in that it is, in a sense, ad hoc: we must manufacture noncompact separating hypersurfaces rather than control the collar-gluing via a Morse function (which is typically proper so that we can control compactness types) as we do in the other cases. Figure \ref{fig:tube} illustrates the simplest such construction.
\end{rem}

\begin{rem}[what is not claimed]
Let us emphasize what we do not claim here. Instantons are saddle points of Euclidean field equations, usually related to Lorentzian solutions by analytic continuation or Wick rotation, and known continuations tend to rely on high symmetry or on a real section of a complexified solution. We impose no field equations and do not attempt to replace or generalize such constructions. The theorem provides a selection rule in the strict sense: if the relative Euler class of the intended Lorentzian region does not vanish, or if the corresponding relations described in Remark~\ref{F1FECTV} fail, then no continuation of the type considered here exists. This is consistent with \cite{Kothawala}, which emphasizes that Wick rotation in curved spacetime is not, in general, equivalent to a change of signature without further assumptions on analyticity and causal-structure. Consequently, if no suitable triple $(g,f,V)$ exists, then the Riemannian manifold in question cannot serve, within this framework, as the precursor of a Lorentzian universe.
\end{rem}

\subsection{Preliminaries}\label{subsec:Preliminaries}

All manifolds, maps and tensor fields are smooth; manifolds are Hausdorff, second countable and without boundary unless we state otherwise. An \emph{open} manifold is noncompact and without boundary. We use the terminology of a signature-changing manifold from~\cite{Hasse + Rieger-Transformation, Hasse + Rieger-Local Transformation,Kossowski - Fold Singularities in Pseudo-Riemannian Geodesic Tubes,Kossowski - Pseudo-Riemannian Metric Singularities and the Extendability of Parallel Transport,Kossowksi + Kriele - Signature type change and absolute time in general relativity,Kossowski + Kriele - Transverse type changing pseudo Riemannian metrics with smooth curvature,Hasse + Rieger-Loops}.

\begin{defn}\label{def:ttcm}
Let $M$ be a smooth manifold without boundary of dimension $\dim M\geq2$, and let $\tilde{g}$ be a smooth symmetric $(0,2)$-tensor field on $M$. We call $(M,\tilde{g})$ a  \emph{transverse type-changing singular semi-Riemannian manifold} if the degeneracy locus 
\[
\mathcal{H:=}\{q\in M\!\!:\tilde{g}\!\!\mid_{q}\textrm{is\;degenerate}\}=\{q\in M\!\!:\det([\tilde{g}_{\mu\nu}](q))=0\}
\]
is non-empty and $\tilde{g}$ is a corank-$1$ transverse type-changing metric, meaning that for every $q\in\mathcal{H}$ and for any local coordinate system around $q$, one has
\[
d(\det[\tilde{g}_{\mu\nu}])_{q}\neq0.
\]
\end{defn}

Thus $\mathcal{\mathcal{H}\subset M}$ is the locus where the rank of $\tilde{g}$ fails to be maximal and where the bilinear type of $\tilde{g}$ changes. Moreover, $\mathcal{H}$ is a smoothly embedded hypersurface of $M$. At each point $q\in\mathcal{H}$, the radical is the subspace
\[
\textrm{Rad}_{q}(\tilde{g}):=\{w\in T_{q}M\mid\tilde{g}_{q}(w,v)=0\text{ for all }v\in T_{q}M\}\subset T_{q}M.
\]
In the transverse type-changing case considered here, this radical is one-dimensional. It may be either tangent or transverse to $\mathcal{H}$. In quantum cosmology applications, the transverse case is of particular interest because it models a smooth transition from a Riemannian region to a Lorentzian spacetime without treating the transition hypersurface as a singular past boundary. 

\begin{defn}
The radical of a metric $\tilde{g}$ is said to be \emph{transverse} to $\mathcal{H}$
if, for every $q\in\mathcal{H}$, $\textrm{Rad}_{q}(\tilde{g})\nsubseteq T_{q}\mathcal{H}$. Equivalently, 
\[
\textrm{Rad}_{q}(\tilde{g})\oplus T_{q}\mathcal{H}=T_{q}M\text{ for all }q\in\mathcal{H}.
\]
\end{defn}

\begin{defn}\label{def:collargluing}
A \emph{collar-gluing} of a connected manifold $M$ is a decomposition $M=\cMR\cup_{\HH}\cML$, where $\HH\subset M$ is a non-empty closed embedded hypersurface with $M\setminus\HH=\Rreg\sqcup\Lreg$ for disjoint non-empty open sets $\Rreg,\Lreg$ with closures $\cMR=\Rreg\cup\HH$, $\cML=\Lreg\cup\HH$. The open sets \(M_R\) and \(M_L\) are not required to be connected.

Such a \emph{collar-gluing is admissible} if there are $f\in C^{\infty}(M)$ and $V\in\Gamma(TM)$ for which $\gt=g-fV^{\flat}\otimes V^{\flat}$ is Riemannian on $\Rreg$, Lorentzian of index one on $\Lreg$, degenerate of corank one precisely along $\HH$, and a transverse type-changing metric with transverse radical along $\HH$.
\end{defn}

\begin{rem}\label{rem:twosided}
A separating hypersurface is two-sided, hence carries a tubular neighbourhood $c:\HH\times\mathbb{R}\xrightarrow{\cong}\mathcal{U}\subseteq M$ with $c(x,0)=x$, $c(\HH\times(0,\infty))\subset\Lreg$ and $c(\HH\times(-\infty,0))\subset\Rreg$. Indeed, let $\HH_0$ be a component of $\HH$ and $N$ a tubular neighbourhood of $\HH_0$ disjoint from the closed set $\HH\setminus\HH_0$. If $\HH_0$ were one-sided $N\setminus\HH_0$ would be connected and hence contained in one of $\Rreg,\Lreg$, contradicting $\HH_0=\cMR\cap\cML$. For the same reason the two sides of $N\setminus\HH_0$ lie in different regions, which fixes the orientation of $c$.

The requirement that $\cMR$ and $\cML$ be the closures of $\Rreg$ and $\Lreg$, and not merely the sets $\Rreg\cup\HH$ and $\Lreg\cup\HH$, cannot be dropped: in the open M\"obius band, let $\HH=\HH_{0}\sqcup\HH_{1}$ be the core circle $\HH_{0}$ together with a circle $\HH_{1}$ parallel to the boundary. Then $M\setminus\HH$ consists of two open annuli, and taking $\Rreg$ to be the outer and $\Lreg$ the inner one gives a decomposition into two disjoint open regions. But $\HH_{0}$ is one-sided and lies in $\overline{\Lreg}$ only, so that $\overline{\Rreg}=\Rreg\cup\HH_{1}\neq\Rreg\cup\HH$ (Figure~\ref{fig:mobius}).
\end{rem}

\begin{figure}[htbp]
  \centering
  \begin{tikzpicture}[line join=round,line cap=round,scale=0.78]
    \def\w{8.0}
    \fill[riem] (0,-1.8) rectangle (\w,1.8);
    \fill[lor]  (0,-0.9) rectangle (\w,0.9);
    \draw[line width=0.6pt,dashed] (0,1.8) -- (\w,1.8);
    \draw[line width=0.6pt,dashed] (0,-1.8) -- (\w,-1.8);
    \draw[line width=0.9pt] (0,-1.8) -- (0,1.8);
    \draw[line width=0.9pt] (\w,-1.8) -- (\w,1.8);
    \draw[-{Stealth[length=6pt]},line width=0.9pt] (0,-0.45) -- (0,0.45);
    \draw[-{Stealth[length=6pt]},line width=0.9pt] (\w,0.45) -- (\w,-0.45);
    \draw[line width=1.2pt] (0,0) -- (\w,0);
    \draw[line width=1.2pt] (0,0.9) -- (\w,0.9);
    \draw[line width=1.2pt] (0,-0.9) -- (\w,-0.9);
    \node at (2.0,1.35) {$\Rreg$};
    \node at (6.0,-1.35) {$\Rreg$};
    \node at (2.0,0.45) {$\Lreg$};
    \node at (6.0,-0.45) {$\Lreg$};
    \node[anchor=west] at (\w+0.25,0) {$\HH_{0}$ (core)};
    \node[anchor=west] at (\w+0.25,0.9) {$\HH_{1}$};
    \node[anchor=west] at (\w+0.25,-0.9) {$\HH_{1}$};
  \end{tikzpicture}
  \caption{The M\"obius band of Remark \ref{rem:twosided}.}
  \label{fig:mobius}
\end{figure}
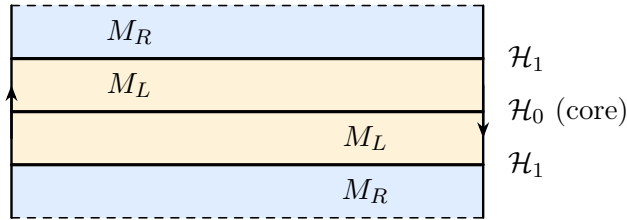

\section{Preliminary results}\label{sec:Preliminary results}

\begin{lem}\label{lem:master}
Let $(M,g)$ be Riemannian and let $M=\cMR\cup_{\HH}\cML$ be a collar-gluing. Let $f\in C^{\infty}(M)$ and $V\in\Gamma(TM)$ be given which satisfy
\begin{enumerate}
  \item[(i)] $\Rreg=\{f<1\}$, $\HH=\{f=1\}$ and $\Lreg=\{f>1\}$.
  \item[(ii)] $g(V,V)=1$ on an open neighbourhood of $\cML$ in $M$, and $g(V,V)\leq1$ on $\Rreg$.
  \item[(iii)] $df(V)\neq0$ at every point of $\HH$.
\end{enumerate}
Then $\gt:=g-fV^{\flat}\otimes V^{\flat}$ is Riemannian on $\Rreg$, Lorentzian of index one on $\Lreg$ and degenerate of corank one on $\HH$. It is a transverse type-changing metric, and its radical along $\HH$ is $\Rad_{p}(\gt)=\mathbb{R}V_{p}$, which is transverse to $\HH$. In particular the collar-gluing is admissible.
\end{lem}

\begin{proof}
Let $p\in M$ and put $\varphi:=1-f\cdot g(V,V)\in C^{\infty}(M)$. If $V_{p}=0$ then $\gt_{p}=g_{p}$ and $\varphi(p)=1$. If $V_{p}\neq0$, set $e_{1}:=V_{p}/\|V_{p}\|_{g}$ and complete to a $g$-orthonormal basis of $T_{p}M$; since $V^{\flat}_{p}=\|V_{p}\|_{g}e_{1}^{\flat}$, the matrix of $\gt_{p}$ is $\operatorname{diag}(\varphi(p),1,\dots,1)$. So in all cases $\gt_{p}$ is positive definite, degenerate of corank one, or Lorentzian of index one according as $\varphi(p)$ is positive, zero or negative. In any chart the matrix determinant lemma gives
\begin{equation}\label{eq:detphi}
\det[\gt_{\mu\nu}]=\varphi\cdot\det[g_{\mu\nu}].
\end{equation}

On $\Lreg$ we have $f>1$ and $g(V,V)=1$ by (ii), so $\varphi=1-f<0$. On $\HH$ we have $f=1$ and $g(V,V)=1$, so $\varphi=0$, and the radical is $\mathbb{R}V_{p}$. On $\Rreg$ we have $f<1$ and $0\leq g(V,V)\leq1$: if $f(p)\geq0$ then $f(p)g(V,V)_{p}\leq f(p)<1$, and if $f(p)<0$ then $f(p)g(V,V)_{p}\leq0<1$; in both cases $\varphi(p)>0$. Thus the degeneracy locus of $\gt$ is exactly $\HH$, with the stated types on either side.

For transversality of the type change: let $p\in\HH$. By~\eqref{eq:detphi} and $\varphi(p)=0$ we get $d(\det[\gt_{\mu\nu}])_{p}=\det[g_{\mu\nu}](p)\cdot d\varphi_{p}$, and $\det[g_{\mu\nu}](p)>0$. By (ii) the function $g(V,V)$ is constantly $1$ on an open neighbourhood of $\HH$, so $\varphi=1-f$ there and $d\varphi_{p}=-df_{p}$, which is nonzero by (iii).

For transversality of the radical: since $df_{p}\neq0$ and $\HH=f^{-1}(1)$ is an embedded hypersurface, $T_{p}\HH=\ker df_{p}$. By (iii), $df_{p}(V_{p})\neq0$, so $V_{p}\notin T_{p}\HH$.
\end{proof}

Condition (iii) says precisely that $V$ is transverse to $\HH$. Condition (ii) ensures that the eigenvalue in the $V$-direction crosses zero exactly where $f$ crosses $1$. The lemma imposes no condition on $V$ away from $\cML$ and its collar, where $V$ may vanish. 

Figure \ref{fig:master} illustrates our main construction which is ultimately based on this lemma.

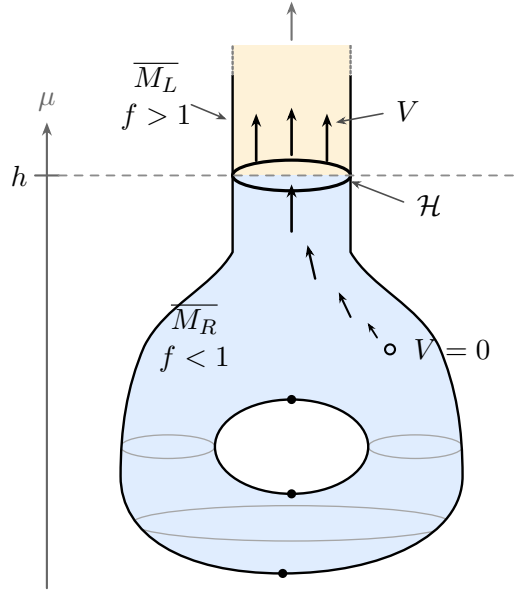
\begin{figure}[htbp]
  \centering
  \begin{tikzpicture}[line join=round,line cap=round,scale=0.78]
    \def\hlev{7.5}
    \def\bodypath{(4.8,0.75)
      .. controls (3.1,0.75)  and (2.05,1.4)  .. (2.05,2.4)
      .. controls (2.05,3.2)  and (2.15,3.9)  .. (2.45,4.6)
      .. controls (2.8,5.3)   and (3.6,5.6)   .. (3.95,6.2)
      .. controls (3.95,6.8)  and (3.95,7.4)  .. (3.95,8.2)
      .. controls (3.95,9.0)  and (3.95,9.6)  .. (3.95,10.6)
      -- (5.95,10.6)
      .. controls (5.95,9.6)  and (5.95,9.0)  .. (5.95,8.2)
      .. controls (5.95,7.4)  and (5.95,6.8)  .. (5.95,6.2)
      .. controls (6.3,5.6)   and (7.1,5.3)   .. (7.45,4.6)
      .. controls (7.75,3.9)  and (7.85,3.2)  .. (7.85,2.4)
      .. controls (7.85,1.4)  and (6.5,0.75)  .. (4.8,0.75) -- cycle}
    \def\holepath{(4.95,2.9) ellipse (1.3 and 0.8)}

    \begin{scope}
      \clip (0,0.2) rectangle (11,9.7);
      \begin{scope}[even odd rule]
        \clip \bodypath \holepath;
        \fill[lor]  (0,0.2) rectangle (11,9.7);
        \fill[riem] (0,0.2) rectangle (11,\hlev);
        \draw[black!35,line width=0.5pt] (4.95,1.6) ellipse (2.7 and 0.3);
        \draw[black!35,line width=0.5pt] (2.85,2.9) ellipse (0.8 and 0.2);
        \draw[black!35,line width=0.5pt] (7.05,2.9) ellipse (0.8 and 0.2);
      \end{scope}
      \begin{scope}
        \clip (0,0.2) rectangle (11,9.2);
        \draw[line width=0.9pt] \bodypath;
      \end{scope}
      \begin{scope}
        \clip (0,9.2) rectangle (11,9.7);
        \draw[black!50,densely dotted,line width=0.9pt] \bodypath;
      \end{scope}
      \draw[line width=0.9pt] \holepath;
      \draw[line width=1.2pt] (4.95,\hlev) ellipse (1.0 and 0.26);
      \fill (4.8,0.75)  circle (2.2pt);
      \fill (4.95,2.1)  circle (2.2pt);
      \fill (4.95,3.7)  circle (2.2pt);
      \draw[-{Stealth[length=5pt]},line width=1.0pt] (4.35,7.75) -- (4.35,8.55);
      \draw[-{Stealth[length=5pt]},line width=1.0pt] (4.95,7.85) -- (4.95,8.65);
      \draw[-{Stealth[length=5pt]},line width=1.0pt] (5.55,7.75) -- (5.55,8.55);
      \draw[-{Stealth[length=5pt]},line width=1.0pt] (4.95,6.55) -- (4.95,7.35);
      \draw[-{Stealth[length=4.5pt]},line width=0.9pt] (5.35,5.75) -- (5.22,6.32);
      \draw[-{Stealth[length=4pt]},line width=0.8pt]  (5.95,5.1)  -- (5.75,5.52);
      \draw[-{Stealth[length=3.5pt]},line width=0.7pt](6.4,4.75)  -- (6.24,5.0);
      \draw[line width=0.8pt] (6.62,4.55) circle (2.4pt);
    \end{scope}
    \draw[black!55,-{Stealth[length=5pt]},line width=0.8pt] (4.95,9.8) -- (4.95,10.45);
    \draw[black!60,-{Stealth[length=5pt]},line width=0.8pt] (0.8,0.5) -- (0.8,8.4);
    \node[black!60] at (0.8,8.75) {$\mu$};
    \draw[black!60,line width=0.8pt] (0.6,\hlev) -- (1.0,\hlev);
    \node at (0.33,\hlev) {$h$};
    \draw[black!45,dashed,line width=0.7pt] (1.0,\hlev) -- (8.8,\hlev);
    \node[align=center] at (2.65,8.85) {$\cML$\\[1pt] $f>1$};
    \draw[black!70,-{Stealth[length=4pt]},line width=0.6pt] (3.25,8.7) -- (3.85,8.45);
    \node[align=center] at (3.3,4.75) {$\cMR$\\[1pt] $f<1$};
    \node[anchor=west] at (6.55,8.6) {$V$};
    \draw[black!70,-{Stealth[length=4pt]},line width=0.6pt] (6.5,8.55) -- (5.7,8.4);
    \node[anchor=west] at (6.9,7.0) {$\HH$};
    \draw[black!70,-{Stealth[length=4pt]},line width=0.6pt] (6.85,7.05) -- (5.95,7.4);
    \node[anchor=west] at (6.85,4.55) {$V=0$};
  \end{tikzpicture}
  \caption{The main construction with $\HH$ connected. A proper Morse function $\mu$ and a regular value $h$  give $\cMR=\mu^{-1}([v_{1},h])$ and $\cML=\mu^{-1}([h,\infty))$, and a shift $f=\rho\circ\mu$ with $\rho$ increasing through $1$ at $h$ satisfies (i) of Lemma \ref{lem:master}. Along $\HH$ the field $V$ is the unit field dual to $d\mu$, transverse to $\HH$ and pointing into $\cML$ because $\mu$ increases there, making the transition a big bang. Over $\cML$ the field has to be extended without zeros, and this is the only obstruction (Remark \ref{rem:euler}, Proposition \ref{prop:euler}). Over $\cMR$ it may be extended in any way with $g(V,V)\leq1$, e.g.\ damped to zero away from a collar of $\HH$.}
  \label{fig:master}
\end{figure}

\subsection{Existence of the transformation function}

\begin{lem}\label{lem:defining}
Every collar-gluing $M=\cMR\cup_{\HH}\cML$ admits an $f\in C^{\infty}(M)$ with $\Rreg=\{f<1\}$, $\HH=\{f=1\}$ and $\Lreg=\{f>1\}$, and $df\neq0$ along $\HH$.
\end{lem}
\begin{proof}
Let $c:\HH\times\mathbb{R}\xrightarrow{\cong}\mathcal{U}$ be as in Remark \ref{rem:twosided} and let $\theta:\mathcal{U}\to\mathbb{R}$ be the second coordinate. Choose $\lambda\in C^{\infty}(M,[0,1])$ with $\operatorname{supp}\lambda\subset\mathcal{U}$ and $\lambda\equiv1$ on a neighbourhood of $\HH$, and define $\varepsilon:M\setminus\HH\to\{\pm1\}$ by  $\varepsilon|_{\Rreg}=-1$, $\varepsilon|_{\Lreg}=+1$. Since $1-\lambda$ vanishes near $\HH$, the product $(1-\lambda)\varepsilon$ extends smoothly by $0$ across $\HH$, and we may set
\[
  f:=1+\lambda\theta+(1-\lambda)\varepsilon\in C^{\infty}(M),
\]
with $\lambda\theta$ understood to be $0$ outside of $\mathcal{U}$. We have $f>1$ on $\Lreg$, and symmetrically $f<1$ on $\Rreg$. Near $\HH$ we have $\lambda\equiv1$, so $f=1+\theta$ and $df=d\theta\neq0$.
\end{proof}

In other words, no global construction is needed for $f$. One may of course obtain $f$ by composing a global Morse function with a reparametrization of $\mathbb{R}$, as suggested in~\cite{Hasse + Rieger-Transformation}. The global function is, however, never needed for the metric: only the behaviour of $f$ in a bicollar of $\HH$ matters, and there any defining function will do. Nonetheless, we use Morse theory below for two other purposes: to produce hypersurfaces with prescribed compactness properties, and to control where topology change occurs.

\subsection{Existence of the vector field}

\begin{rem}[The  relative Euler class]\label{rem:euler}
Let $X$ be a smooth $n$-manifold with boundary and let $v$ be a nowhere-vanishing vector field along $\partial X$, everywhere transverse to it. Such $v$ exist, and any two which point the same way along each boundary component are homotopic through transverse fields. The complete obstruction to extending $v$ to a nowhere-vanishing vector field on $X$ is the relative Euler class $e(TX,v)\in H^{n}(X,\partial X;\mathbb{Z}_{w})$, the coefficients being twisted by the first Stiefel--Whitney class $w_{1}(TX)$, which vanishes iff $X$ is orientable \cite{MilnorStasheff}.\footnote{Given a section $v$ of $S(TX)|_{\partial X}$ the restriction of the sphere bundle of the tangent bundle to the boundary, the obstructions to extending it over $(X,\partial X)$ lie in $H^{k+1}(X,\partial X;\pi_{k}(S^{n-1}))$. Since $S^{n-1}$ is $(n-2)$-connected, only $k=n-1$ contributes.} We record its two evaluations.

\emph{$X$ noncompact:} If every component of $X$ is noncompact then \[ H^{n}(X,\partial X;\mathbb{Z}_{w})\cong H^{\mathrm{lf}}_{0}(X)=0 \] by Lefschetz duality with locally-finite (Borel--Moore) homology. On a connected noncompact space we have $H^{\mathrm{lf}}_{0}=0$, so the obstruction vanishes.

\emph{$X$ compact:} If $X$ is compact and connected then $H^{n}(X,\partial X;\mathbb{Z}_{w})\cong\mathbb{Z}$, and under this isomorphism $e(TX,v)$ is the sum of the indices at the zeroes of a generic extension of $v$. By the Poincar\'e--Hopf theorem with boundary conditions in the form due to Morse~\cite{Morse1929} (see also~\cite{Pugh1968}), we have
\begin{equation}\label{eq:morse}
e(TX,v)=\chi(X)-\chi(\partial_{-}X),
\end{equation}
where $\partial_{-}X\subseteq\partial X$ is the union of the boundary components along which $v$ points \emph{into} $X$.\footnote{Indeed, Pugh's formula \cite[\S1]{Pugh1968} states that the index equals
\(
  \chi(X)-\chi(\partial X)+\sum_{i\geq1}\bigl(\chi(R_{-}^{i})-\chi(\Gamma^{i})\bigr),
\)
where $R_{-}^{1}\subseteq\partial X$ is the closure of the set along which $v$ points out of $X$, $\Gamma^{1}\subseteq\partial X$ is the locus where $v$ is tangent to $\partial X$, and $R_{-}^{i},\Gamma^{i}\subseteq\Gamma^{1}$ for $i\geq2$ \cite[\S2]{Pugh1968}. (Pugh writes $\chi(A,B)=\chi(A)-\chi(B)$.) For $v$ transverse to $\partial X$ we have $\Gamma^{1}=\emptyset$, so that the genericity hypothesis of \cite[\S5]{Pugh1968} is vacuous, and $R_{-}^{i}=\Gamma^{i}=\emptyset$ for $i\geq2$, and $R_{-}^{1}=\partial_{+}X:=\partial X\setminus\partial_{-}X$. Since $\partial X=\partial_{-}X\sqcup\partial_{+}X$, the sum becomes
\(
  \chi(X)-\bigl(\chi(\partial_{-}X)+\chi(\partial_{+}X)\bigr)+\chi(\partial_{+}X)=\chi(X)-\chi(\partial_{-}X),
\)
giving \eqref{eq:morse}.} 
In particular $e(TX,v)=\chi(X)$ for a uniformly outward $v$ and $(-1)^{n}\chi(X)$ for a uniformly inward one, using that $\chi(\partial X)=(1+(-1)^{n-1})\chi(X)$ for compact $X$ \cite[Proposition 18.6.2]{tomDieck} (from which one derives that manifolds of odd Euler characterstic do not bound).

We should note two practical points. First, it is the specific requirement that $v$ itself be extended into the interior which may be obstructed, not that there exist some nowhere-vanishing vector field on $X$: the closed $n$-disc has a nowhere-vanishing vector field, but not one which is inward-pointing along $S^{n-1}$. Second, if $\partial X$ is disconnected then the choice of $\partial_{-}X$ matters, and \eqref{eq:morse} can vanish for a compact $X$ with $\chi(X)\neq0$. For $n$ even this cannot happen, $\partial X$ being closed and odd-dimensional,\footnote{Closed odd-dimensional manifolds have $\chi=0$ by Poincar\'e duality.} so that $\chi(\partial_{-}X)=0$ and $e(TX,v)=\chi(X)$ for every choice. For $n$ odd it does happen: see Example \ref{ex:caps}.
\end{rem}

\begin{prop}\label{prop:euler}
Let $(M,g)$ be Riemannian and let $M=\cMR\cup_{\HH}\cML$ be a collar-gluing. The following are equivalent.
\begin{enumerate}
  \item The collar-gluing is admissible.
  \item $\cML$ carries a nowhere-vanishing vector field transverse to
$\HH=\partial\cML$ along $\HH$.
  \item For some transverse boundary field $\nu$ along $\HH$, the relative Euler
class $e(T\cML,\nu)\in H^{n}(\cML,\HH;\mathbb{Z}_{w})$ vanishes.
  \item Every compact component $X$ of $\cML$ admits a decomposition
$\partial X=\partial_{+}X\sqcup\partial_{-}X$ into unions of connected components with $\chi(X)=\chi(\partial_{-}X)$.
\end{enumerate}
If $\dim M$ is even, or if $\nu$ is required to point uniformly into $\Lreg$ (or uniformly into $\Rreg$) along $\HH$, then (4) simplifies to: every compact component of $\cML$ has vanishing Euler characteristic.
\end{prop}

\begin{proof}
$(2)\Leftrightarrow(3)\Leftrightarrow(4)$ is Remark~\ref{rem:euler} applied componentwise.

$(1)\Rightarrow(2)$: Let $(f,V)$ realize the admissibility. At $p\in\cML$ the tensor $\gt_{p}$ is not positive definite so $V_{p}\neq0$. Along $\HH$ the radical is spanned by $V$ and is transverse to $\HH$ by hypothesis, so $V|_{\cML}$ is as required.

$(2)\Rightarrow(1)$: Let $U$ be nowhere vanishing on $\cML$ and transverse to $\HH$. Since $\cML\subset M$ is a closed domain with smooth boundary, $U$ extends to a smooth vector field on an open neighbourhood of $\cML$: in boundary charts extend the components by Seeley's extension theorem~\cite{Seeley} and glue with a partition of unity. Deleting the zero set of the extension, which is closed and disjoint from $\cML$, and normalizing, we obtain an open neighbourhood $W\supseteq\cML$ and a $g$-unit field $\widehat U$ on $W$ with $\widehat U|_{\HH}=U|_{\HH}$, hence transverse to $\HH$. Choose $\psi\in C^{\infty}(M,[0,1])$ with $\operatorname{supp}\psi\subset W$ and $\psi\equiv1$ on an open neighbourhood $W'\subseteq W$ of $\cML$, and set $V:=\psi\widehat U$, extended by $0$. Then $V$ is smooth, $g(V,V)=1$ on $W'$ and $g(V,V)=\psi^{2}\leq1$ everywhere. Take $f$ from Lemma~\ref{lem:defining}. Along $\HH$ we have $V=\widehat U$ transverse to $\HH=f^{-1}(1)$ and $df\neq0$, so $df(V)\neq0$ there. Now Lemma \ref{lem:master} concludes.

The final statement follows from the same Remark: for $n$ even we have $\chi(\partial_{-}X)=0$ for every choice, and for uniform $\nu$ we have $\partial_{-}X=\emptyset$ or $\partial_{-}X=\partial X$, where $\chi(X)-\chi(\partial X)=(-1)^{n}\chi(X)$.
\end{proof}

\begin{cor}\label{cor:gindep}
Admissibility does not depend on $g$: if a collar-gluing is admissible for one Riemannian metric on $M$, it is admissible for every one.
\end{cor}

\begin{proof}
Condition (2) of Proposition~\ref{prop:euler} is independent of $g$.
\end{proof}

Thus the Riemannian metric enters only through the tensor $\gt$ itself and has no bearing on whether such a tensor exists. Proposition~\ref{prop:euler} is the first half of the main theorem, and it determines the shape of the second. Since the criterion refers only to $\cML$, and is vacuous unless $\cML$ has a compact component, it remains to decide which compactness types of collar-gluing a given $M$ admits. For how the obstruction spills over into the Riemannian region, see Corollary \ref{cor:MV}.

\section{Geometro-topological selection rules}

\subsection{Preliminary lemmata}
Morse theory is typically applied in the compact setting, where there are finitely many critical points. The underlying machinery, however, does not require compactness. What is lost is the finiteness of the critical set but what remains (and what we use) is properness.

\begin{lem}\label{lem:propermorse}
Every smooth manifold $M$ admits a proper Morse function $\mu:M\to[0,\infty)$, and such a $\mu$ attains a minimal critical value.
\end{lem}

\begin{proof} We will simply observe that the standard Morse function on an arbitrary manifold is already of this type.
Embed $M\hookrightarrow\mathbb{R}^{N}$ as a closed submanifold and put $\mu(x)=|x-p|^{2}$. By \cite[Thm.~6.6]{Milnor1963} this is Morse for almost every $p$ and it is proper because $\mu^{-1}[0,c]=M\cap\overline{B}(p,\sqrt{c})$ is closed and bounded. Properness gives the second assertion: critical points are isolated and each compact sublevel set contains finitely many, so the critical values form a discrete closed subset of $[0,\infty)$ and have a least element. The set is non-empty because a proper function bounded below attains a global minimum.
\end{proof}

\begin{lem}[compact core]\label{lem:core}
Let $M$ be a connected open manifold and $C\subset M$ compact. Then there is a compact codimension-zero submanifold $K\subset M$ with smooth non-empty boundary such that $C\subset\intr(K)$ and every component of $M\setminus\intr(K)$ is noncompact.
\end{lem}

\begin{proof}
Figure \ref{fig:core} illustrates the idea. Take $\mu$ as in Lemma~\ref{lem:propermorse} with least critical value $v_{1}$, choose a regular value $h>\sup_{C}\mu$ and set $K_{0}:=\mu^{-1}[v_{1},h]$, a compact codimension-zero submanifold with smooth boundary $\mu^{-1}(h)$ containing $C$ in its interior. We have $\partial K_{0}\neq\emptyset$ since $M$ is connected and noncompact. As $\partial K_{0}$ is compact, it has finitely many components, and every component of $M\setminus\intr(K_{0})$ has boundary a non-empty union of them. Thus $M\setminus\intr(K_{0})$ has finitely many components. Let $K$ be the union of $K_{0}$ with the compact ones. Then $K$ is compact with smooth boundary, and $M\setminus\intr(K)$ is the union of the noncompact ones, which is non-empty because $M$ is not compact.
\end{proof}

\begin{rem}\label{rem:noremorse}
	It is really only the properness of $\mu$ that Lemma \ref{lem:core} uses and not that it is Morse.
\end{rem}

\subsection{$M_{R}$ with compact closure}\label{subsec:cMR compact}

\begin{thmstar}[Case A]
\label{thm:A}
Let $(M,g)$ be a connected open Riemannian manifold of dimension $n\geq2$. Then there exists an admissible collar gluing $M=\cMR\cup_{\HH}\cML$ with $\cMR$ compact, with $\HH$ compact, and with every component of $\cML$ noncompact.
\end{thmstar}

At a formal level, the proof is a combination of
Lemma~\ref{lem:core}, Remark~\ref{rem:euler}, and
Proposition~\ref{prop:euler}. For completeness, however, we collect the
relevant ingredients once more. The proofs of the remaining cases will
be given in less detail.

\begin{proof}
Let $\mu\colon M\to[0,\infty)$ be a proper Morse function and let $v_1=\min_M\mu$ be its global minimum (Lemma \ref{lem:propermorse}). Every sublevel set $\mu^{-1}([v_1,h])\subset M$ is compact for all \(h\geq v_1\) due to properness. It is really only the properness of $\mu$ that we will use.

Let $h>v_1$ be a regular value of $\mu$ and take 
\[
	K_0=\mu^{-1}[v_1,h].
\] 
Now, the complement $M\setminus \operatorname{int} K_0 = \mu^{-1}[h,\infty)$ may have compact connected components. As in the proof of Lemma \ref{lem:core}, we define $\cMR$ to be the union of $K_0$ and all compact connected components of $\mu^{-1}[h,\infty)$: see Figure \ref{fig:core}. Since $\partial K_0=\mu^{-1}\{h\}$ is compact by properness it has finitely many components, whence $\mu^{-1}[h,\infty)$, having the same boundary (and $M$ being connected), has finitely many components. Therefore $\cMR$ is compact. We write $\HH=\partial\cMR$ and $M_L=M\setminus\cMR$, so $\cML=M_L\cup \HH$. All components of $\cML$ are noncompact by construction.

Figure \ref{fig:humps} illustrates how we cannot in general get around adding these components to $K_0$ by hand by choosing $h$ high enough.

As for admissibility, we are essentially in the situation illustrated in
Figure~\ref{fig:master}. We take $f$ from Lemma \ref{lem:defining} for this collar-gluing and take $\nu=\nabla f$ near $\HH$, where $f\equiv 1+\theta$, so $\nu$ is transverse to it (see the proof of Lemma \ref{lem:defining}).  
Because $\cML$ is noncompact, the relative Euler class of $\cML$ with respect to $\nu$ vanishes by Remark \ref{rem:euler}. Hence $\nu$ extends to a nowhere-vanishing vector field $V$ on all of $\cML$.  As in the proof
of Proposition~\ref{prop:euler}, $V$ may also be extended into $M_R$. Finally, Lemma \ref{lem:master} yields admissibility for $\tilde{g}=g-fV^\flat\otimes V^\flat$.
\end{proof}

This is the configuration of the no-boundary proposal, and it is available on every open manifold without further conditions. Alternatively, by Lemma~\ref{lem:defining}, one may take any proper $f\in C^{\infty}(M)$ instead of a proper Morse function, as the latter one is a convenience, not a necessity.

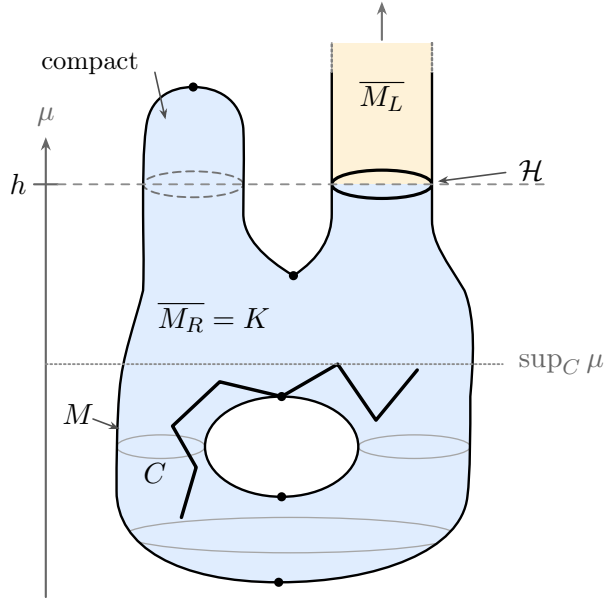
\begin{figure}[htbp]
  \centering
  \begin{tikzpicture}[line join=round,line cap=round,scale=0.78]
    \def\hlev{7.5}
    \def\bodypath{(4.7,0.75)
      .. controls (3.0,0.75)  and (1.95,1.35) .. (1.95,2.3)
      .. controls (1.95,2.8)  and (1.95,3.4)  .. (2.0,4.0)
      .. controls (2.05,4.6)  and (2.25,5.1)  .. (2.4,5.7)
      .. controls (2.42,6.2)  and (2.4,6.6)   .. (2.4,7.0)
      .. controls (2.4,7.6)   and (2.4,8.1)   .. (2.45,8.45)
      .. controls (2.5,8.95)  and (2.85,9.15) .. (3.25,9.15)
      .. controls (3.65,9.15) and (4.05,8.95) .. (4.1,8.45)
      .. controls (4.12,8.0)  and (4.1,7.5)   .. (4.1,7.0)
      .. controls (4.1,6.6)   and (4.45,6.25) .. (4.95,5.95)
      .. controls (5.4,6.25)  and (5.6,6.6)   .. (5.6,7.0)
      .. controls (5.6,7.6)   and (5.6,8.2)   .. (5.6,8.8)
      .. controls (5.6,9.6)   and (5.6,10.2)  .. (5.6,11.2)
      -- (7.3,11.2)
      .. controls (7.3,10.2)  and (7.3,9.4)   .. (7.3,8.6)
      .. controls (7.3,7.8)   and (7.3,7.3)   .. (7.3,6.9)
      .. controls (7.3,6.5)   and (7.55,6.2)  .. (7.85,5.7)
      .. controls (8.05,5.2)  and (8.0,4.5)   .. (7.95,3.9)
      .. controls (7.95,3.3)  and (7.95,2.6)  .. (7.9,2.2)
      .. controls (7.8,1.3)   and (6.4,0.75)  .. (4.7,0.75) -- cycle}
    \def\holepath{(4.75,3.05) ellipse (1.3 and 0.85)}
    \def\cpath{(3.05,1.85) -- (3.3,2.7) -- (2.9,3.4) -- (3.7,4.15) -- (4.75,3.9)
               -- (5.7,4.45) -- (6.35,3.5) -- (7.05,4.35)}

    \begin{scope}
      \clip (0,0.2) rectangle (11.2,9.9);
      \begin{scope}[even odd rule]
        \clip \bodypath \holepath;
        \fill[lor]  (0,0.2) rectangle (11.2,9.9);
        \fill[riem] (0,0.2) rectangle (11.2,\hlev);
        \fill[riem] (0,\hlev) rectangle (4.85,9.9);
        \draw[black!35,line width=0.5pt] (4.9,1.55) ellipse (2.8 and 0.3);
        \draw[black!35,line width=0.5pt] (2.7,3.05) ellipse (0.75 and 0.2);
        \draw[black!35,line width=0.5pt] (7.0,3.05) ellipse (0.95 and 0.2);
        \draw[black!55,densely dashed,line width=0.7pt] (3.25,\hlev) ellipse (0.85 and 0.22);
        \draw[line width=1.2pt] (6.45,\hlev) ellipse (0.85 and 0.24);
      \end{scope}
      \begin{scope}
        \clip (0,0.2) rectangle (11.2,9.4);
        \draw[line width=0.9pt] \bodypath;
      \end{scope}
      \begin{scope}
        \clip (0,9.4) rectangle (11.2,9.9);
        \draw[black!50,densely dotted,line width=0.9pt] \bodypath;
      \end{scope}
      \draw[line width=0.9pt] \holepath;
      \draw[line width=1.3pt] \cpath;
      \fill (4.7,0.75) circle (2.2pt);
      \fill (4.75,2.2) circle (2.2pt);
      \fill (4.75,3.9) circle (2.2pt);
      \fill (4.95,5.95) circle (2.2pt);
      \fill (3.25,9.15) circle (2.2pt);
    \end{scope}
    \draw[black!55,-{Stealth[length=5pt]},line width=0.8pt] (6.45,9.95) -- (6.45,10.6);
    \draw[black!60,-{Stealth[length=5pt]},line width=0.8pt] (0.75,0.5) -- (0.75,8.3);
    \node[black!60] at (0.75,8.65) {$\mu$};
    \draw[black!60,line width=0.8pt] (0.55,\hlev) -- (0.95,\hlev);
    \node at (0.28,\hlev) {$h$};
    \draw[black!45,dashed,line width=0.7pt] (0.95,\hlev) -- (9.2,\hlev);
    \draw[black!45,densely dotted,line width=0.7pt] (0.75,4.45) -- (8.5,4.45);
    \node[black!60,anchor=west] at (8.6,4.45) {$\sup_{C}\mu$};
    \node at (3.6,5.25) {$\cMR=K$};
    \node at (6.45,9.0) {$\cML$};
    \node[anchor=west] at (8.6,7.75) {$\HH$};
    \draw[black!70,-{Stealth[length=4pt]},line width=0.6pt] (8.55,7.7) -- (7.4,7.55);
    \node at (2.6,2.6) {$C$};
    \node at (1.3,3.6) {$M$};
    \draw[black!70,-{Stealth[length=4pt]},line width=0.6pt] (1.55,3.52) -- (2.0,3.35);
    \node[anchor=east] at (2.5,9.55) {\small compact};
    \draw[black!70,-{Stealth[length=4pt]},line width=0.6pt] (2.6,9.4) -- (2.78,8.6);
  \end{tikzpicture}
  \caption{The construction behind Case (A), with $\mu$ the proper Morse function of Lemma~\ref{lem:propermorse} whose critical points are the filled dots. A regular value $h>\sup_{C}\mu$ splits $M$, as in Lemma~\ref{lem:core}, into the compact sublevel set $K_{0}=\mu^{-1}([v_{1},h])$ and the two components of $\mu^{-1}([h,\infty))$: the cap over the dashed circle, which is compact and which we absorb into $K$ by hand (as in the proof of Lemma~\ref{lem:core}), and the noncompact tube. Hence $\HH=\partial K$ is a single circle and $\cML$ is the tube alone. The potential instanton $\cMR=K$ is compact (and can contain a prescribed compact set $C$). It also contains every critical point of $\mu$, so that the Lorentzian era carries no topology change (see Corollary \ref{cor:selection}).}
  \label{fig:core}
\end{figure}

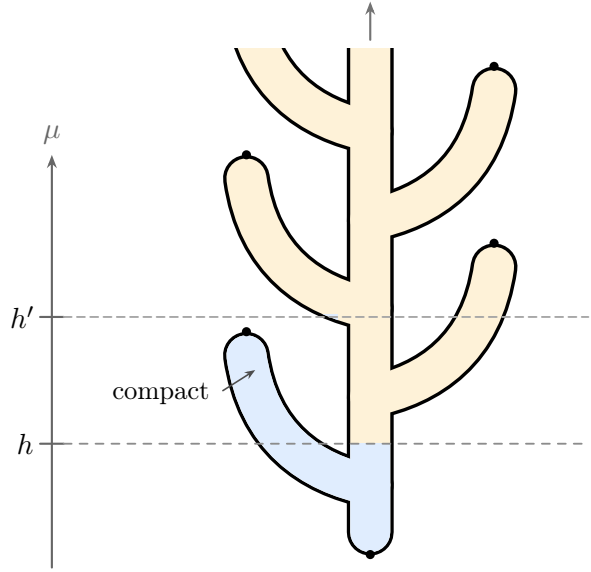
\begin{figure}[htbp]
  \centering
  \begin{tikzpicture}[line join=round,line cap=round,scale=0.78]
    \def\hlev{2.7}
    \def\hhlev{4.85}
    \def\comb{
      \draw[combstyle] (5.9,1.2) -- (5.9,9.6);
      \draw[combstyle] (5.9,2.0) .. controls (4.9,2.2) and (4.0,2.8) .. (3.8,4.2);
      \draw[combstyle] (5.9,3.5) .. controls (6.9,3.7) and (7.8,4.3) .. (8.0,5.7);
      \draw[combstyle] (5.9,5.0) .. controls (4.9,5.2) and (4.0,5.8) .. (3.8,7.2);
      \draw[combstyle] (5.9,6.5) .. controls (6.9,6.7) and (7.8,7.3) .. (8.0,8.7);
      \draw[combstyle] (5.9,8.0) .. controls (4.9,8.2) and (4.0,8.8) .. (3.8,10.2);}

    \begin{scope}
      \clip (0,0.2) rectangle (10.6,9.4);
      \tikzset{combstyle/.style={line width=0.62cm,black}}\comb
      \tikzset{combstyle/.style={line width=0.53cm,lor}}\comb
      \begin{scope}
        \clip (0,0.2) rectangle (10.6,\hlev);
        \tikzset{combstyle/.style={line width=0.53cm,riem}}\comb
      \end{scope}
      \begin{scope}
        \clip (0,\hlev) rectangle (5.35,4.9);
        \tikzset{combstyle/.style={line width=0.53cm,riem}}\comb
      \end{scope}
      \fill (5.9,0.82) circle (2.2pt);
      \foreach \p in {(3.8,4.6),(8.0,6.1),(3.8,7.6),(8.0,9.1)} \fill \p circle (2.2pt);
    \end{scope}
    \draw[black!55,-{Stealth[length=5pt]},line width=0.8pt] (5.9,9.55) -- (5.9,10.2);
    \draw[black!60,-{Stealth[length=5pt]},line width=0.8pt] (0.5,0.6) -- (0.5,7.6);
    \node[black!60] at (0.5,7.95) {$\mu$};
    \draw[black!60,line width=0.8pt] (0.3,\hlev) -- (0.7,\hlev);
    \node at (0.05,\hlev) {$h$};
    \draw[black!45,dashed,line width=0.7pt] (0.7,\hlev) -- (9.6,\hlev);
    \draw[black!60,line width=0.8pt] (0.3,\hhlev) -- (0.7,\hhlev);
    \node at (0.0,\hhlev) {$h'$};
    \draw[black!35,densely dashed,line width=0.7pt] (0.7,\hhlev) -- (9.6,\hhlev);
    \node[anchor=east] at (3.35,3.55) {\small compact};
    \draw[black!70,-{Stealth[length=4pt]},line width=0.6pt] (3.4,3.6) -- (3.95,3.95);
  \end{tikzpicture}
  \caption{Why the passage from $K_{0}$ to $K$ in Lemma~\ref{lem:core} cannot be
replaced by the choice of a higher level. 
}
  \label{fig:humps}
\end{figure}

We now consider the compact--compact case, in which both closed sides $\cMR$, $\cML$ of the signature-changing hypersurface are required to be compact.

\begin{corstar}[Case B]
\label{cor:B}
Let $(M,g)$ be a connected Riemannian manifold of dimension $n\geq2$. A collar gluing $M=\cMR\cup_{\HH}\cML$ has both sides compact if and only if $M$ is compact. If $M$ is compact, the collar-gluing is admissible if and only if the boundary of $\cML$ admits
a decomposition $\partial\cML=\partial_{+}\cML\sqcup\partial_{-}\cML$ into unions of connected components of $\HH$ such that $\chi(\cML)=\chi(\partial_{-}\cML)$. If $n$ is even, this condition reduces to $\chi(\cML)=0$.
\end{corstar}

\begin{proof}
The equivalence statement regarding compactness types is clear since $\cMR$ and $\cML$ are closed. The criterion is Proposition~\ref{prop:euler}.
\end{proof}

The following is an elementary computation using Mayer--Vietoris and shows how the Riemannian region is affected by the Euler characteristic condition. We provide a proof for completeness.

\begin{cor}\label{cor:MV}
Let $M=\cMR\cup_{\HH}\cML$ be a collar-gluing of a closed manifold $M$. Then
\[
  \chi(\cMR)+\chi(\cML)=\chi(M)+\chi(\HH).
\]
If $n$ is even, then $\chi(\HH)=0$; and if further the collar-gluing is admissible, then $\chi(\cMR)=\chi(M)$. If $n$ is odd, then $\chi(M)=0$ and $\chi(\cMR)=\chi(\cML)=\tfrac12\chi(\HH)$.
\end{cor}

\begin{proof}
The spaces $M$, $\cMR$, $\cML$ and $\HH$ are compact manifolds, possibly with boundary, and $\chi$ is the alternating sum of the Betti numbers. Let $c$ be the tubular neighbourhood of Remark~\ref{rem:twosided}, and set
\[
  U_{R}:=\Rreg\cup c\bigl(\HH\times(-\infty,1)\bigr),\qquad U_{L}:=\Lreg\cup c\bigl(\HH\times(-1,\infty)\bigr).
\]
These are open, cover $M$, and meet in $U_{R}\cap U_{L}=c(\HH\times(-1,1))$. The homotopy which scales the collar coordinate of the points of $c(\HH\times[0,1))$ to $0$ and fixes $\cMR$ is continuous, because $\cMR$ and $c(\HH\times[0,1))$ are closed in $U_{R}$ (the latter since $c(\HH\times[0,1])$ is compact). It is a deformation retraction of $U_{R}$ onto $\cMR$. In the same way $U_{L}$ deformation retracts onto $\cML$ and $U_{R}\cap U_{L}$ onto $\HH$. The Mayer--Vietoris sequence (with field coefficients) of the cover $\{U_{R},U_{L}\}$, which thus applies, is the exact sequence
\[
  \cdots\to H_{k}(\HH)\to H_{k}(\cMR)\oplus H_{k}(\cML)\to H_{k}(M)\to H_{k-1}(\HH)\to\cdots,
\]
vanishing in degrees above $n$. The alternating sum of the dimensions in a finite exact sequence is zero, which gives 
\[
\chi(\HH)-\chi(\cMR)-\chi(\cML)+\chi(M)=0.
\]

As mentioned in Remark \ref{rem:euler}, a closed manifold of odd dimension has $\chi=0$ by Poincar\'e duality ($b_k=b_{\dim-k}$). For $n$ even this $\chi(\HH)=0$. If the gluing is moreover admissible, then Case (B) gives $\chi(X)=0$ for each of the finitely many components $X$ of $\cML$, so $\chi(\cML)=0$ and the identity reads $\chi(\cMR)=\chi(M)$. For $n$ odd it gives $\chi(M)=0$. It also applies to the double $D\cML=\cML\cup_{\HH}\cML$, which is a closed $n$-manifold, smooth by means of the collar. Since $\partial\cML=\HH$, the argument above applied to $D\cML$ gives $0=2\chi(\cML)-\chi(\HH)$, and likewise $0=2\chi(\cMR)-\chi(\HH)$.
\end{proof}

For $n$ odd the identity therefore gives no information about the potential instanton beyond $\chi(\cMR)=\chi(\cML)$. For $n$ even the condition $\chi(\cMR)=\chi(M)$ is necessary for admissibility, and it is sufficient when $\cML$ is connected, but not sufficient in general since it only sees the sum $\sum_{X}\chi(X)$ over the components of $\cML$. 

To illustrate: on $M=S^{2}$ it forces $\chi(\cMR)=2$, which is impossible for a connected $\cMR$, since a compact connected surface with non-empty boundary has $\chi\leq1$. An admissible compact--compact gluing of $S^{2}$ must therefore have at least two potential instantons. Two are realized in Example \ref{ex:caps}. For three, let $\cMR$ be a pair of pants together with three discs, and $\cML$ the three cylinders joining each disc to a cuff of the pants. Then $\chi(\cMR)=-1+3=2$, and every component of $\cML$ has $\chi=0$.

\subsection{\(M_R\) with noncompact closure}\label{A227JE6}
The noncompact--noncompact case is governed both by the ends of \(M\) and by whether the transition hypersurface is compact or noncompact; see Figure \ref{fig:compact-noncompact-H}. 

Recall \cite{HughesRanicki, Porter1995} that the \emph{ends} of $M$ are the elements of $\varprojlim_{i}\pi_{0}^{\mathrm{nc}}(M\setminus K_{i})$, the inverse limit over the components with noncompact closure, taken along an exhaustion $K_{1}\subset\intr(K_{2})\subset\cdots$ by compact codimension-zero submanifolds with smooth boundary; each $M\setminus K_{i}$ then has finitely many components, since $\partial K_{i}$ does. This is a purely topological notion which formalises the idea of counting the `directions at infinity' in $M$ by considering compact exhaustions. The simplest examples are as follows: 
\begin{itemize}
	\item $\mathbb{R}$ has two ends.
	\item $S^{n-1}\times\mathbb{R}$, $n\geq2$, likewise has two ends.
	\item $\mathbb{R}^{n}$ has one for $n\geq2$.
\end{itemize}

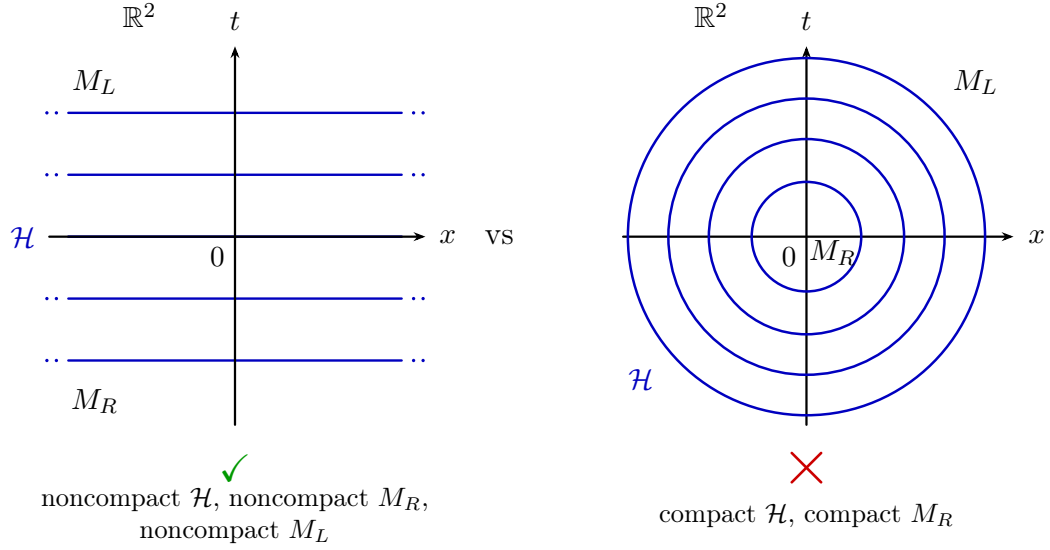
\begin{figure}[htbp]
    \centering
    \begin{tikzpicture}[line join=round,line cap=round,x=0.63cm,y=0.63cm]
  \colorlet{lev}{blue!75!black}

  \begin{scope}
    \coordinate (O) at (5,5);
    \foreach \y in {2.4,3.7,6.3,7.6}{
      \draw[lev,line width=0.9pt] (1.5,\y)--(8.5,\y);
      \foreach \x in {1.05,1.25,8.75,8.95}
        \fill[lev] (\x,\y) circle (0.035);
    }
    \draw[lev,line width=1.0pt] (1.5,5)--(8.5,5);
    \draw[-{Stealth[length=4.5pt]},line width=0.85pt] (1.1,5)--(9.0,5);
    \draw[-{Stealth[length=4.5pt]},line width=0.85pt] (5,1.05)--(5,9.0);
    \node[below left] at (O) {$0$};
    \node[anchor=west] at (9.08,5) {$x$};
    \node[anchor=south] at (5,9.12) {$t$};
    \node at (3,9.65) {$\mathbb{R}^{2}$};
    \node at (2.05,8.25) {$M_L$};
    \node at (2.05,1.55) {$M_R$};
    \node[lev,anchor=east] at (1.02,5) {$\mathcal H$};
    \node[green!60!black,font=\Large] at (5,0.15) {$\checkmark$};
    \node[align=center] at (5,-0.85)
      {\small noncompact $\mathcal H$, noncompact $M_R$,\\[-1pt]
       \small noncompact $M_L$};
  \end{scope}

  \node[font=\large] at (10.55,5.0) {vs};

  \begin{scope}[shift={(12.0,0)}]
    \coordinate (C) at (5,5);
    \draw[-{Stealth[length=4.5pt]},line width=0.85pt] (1.15,5)--(9.35,5);
    \draw[-{Stealth[length=4.5pt]},line width=0.85pt] (5,1.05)--(5,9.0);
    \foreach \r in {1.15,2.05,2.9,3.75}
      \draw[lev,line width=0.95pt] (C) circle (\r);
    \node[below left] at (C) {$0$};
    \node[anchor=west] at (9.45,5) {$x$};
    \node[anchor=south] at (5,9.12) {$t$};
    \node at (3,9.65) {$\mathbb{R}^{2}$};
    \node at (5.55,4.67) {$M_R$};
    \node at (8.55,8.25) {$M_L$};
    \node[lev,anchor=east] at (2.0,2.0) {$\mathcal H$};
    \node[red!80!black,font=\Huge] at (5,0.15) {$\times$};
    \node at (5,-0.85) {\small compact $\mathcal H$, compact $M_R$};
  \end{scope}
\end{tikzpicture}
    \caption{$\mathbb{R}^n$ for $n\geq2$ does not admit an admissible collar-gluing with noncompact $\cMR$ and $\cML$ but compact $\HH$ since it has two ends. 
    The affine hyperplane splitting of \(\mathbb R^n\) belongs to a different subcase as $\HH$ is then noncompact and arises as the level set of a nonproper function, for example a coordinate projection.}
    \label{fig:compact-noncompact-H}
\end{figure}

Thus the noncompact--noncompact case naturally splits into two subcases. When the transition hypersurface \(\mathcal H\) is compact, the relevant
topological condition is that \(M\) have at least two ends. In this case,
the construction can be carried out using a proper Morse function that is
unbounded both above and below. If \(\mathcal H\) is
noncompact, properness must be abandoned in the construction since level sets of proper functions are compact.


\begin{thmstar}[Case C]
\label{thm:C}
Let $(M,g)$ be a connected open Riemannian manifold of dimension \(n\geq2\). Then $M$ admits an admissible collar-gluing with $\HH$ compact and $\cMR,\cML$ both noncompact if and only if $M$ has at least two ends.
\end{thmstar}

\begin{proof}
$\Rightarrow$: Since $\HH$ is compact, $\Rreg$ and $\Lreg$ are noncompact. Choose the exhaustion so that $\HH\subset\intr(K_{1})$ (using a collar). Then for each $i$ every component of $M\setminus K_{i}$ lies wholly in $\Rreg$ or wholly in $\Lreg$. If all (of the finitely many) components contained in $\Rreg$ had compact closure then $\Rreg$ would be compact, so at least one of them has noncompact closure; likewise for $\Lreg$. The two resulting subsystems are non-empty, finite and disjoint, so each has non-empty inverse limit. Thus $M$ has at least two ends.

$\Leftarrow$: We first produce a proper Morse function $\mu:M\to\mathbb{R}$ unbounded above and below. Since $M$ has at least two ends there is a compact codimension-zero $K\subset M$ with smooth boundary such that $M\setminus K$ has at least two components with noncompact closure. Let $V_{+}$ denote one of them, and $V_{-}$ the union of the others. Both are open and noncompact. Now, let $\rho:M\to[1,\infty)$ be a proper smooth exhaustion and $\sigma:M\to[-1,1]$ smooth with $\sigma\equiv1$ on $V_{+}$ and $\sigma\equiv-1$ on $V_{-}$ off a compact collar of $\partial K$. Then $u:=\sigma\rho$ tends to $\pm\infty$ along $V_{\pm}$ and is proper since $\{|\sigma|<1\}$ has compact closure and $|u|=\rho$ outside of it. As Morse functions are dense in the strong topology, choose a Morse function $\mu$ with $\sup_{M}|\mu-u|<1$. The bound preserves properness and unboundedness. (Again Morseness is irrelevant to the statement itself.)

Now let $h$ be a regular value. The manifold $\mu^{-1}([h,\infty))$ is noncompact and has compact boundary $\mu^{-1}(h)$, hence finitely many components. Take $\cML$ to be the union of the noncompact ones. Then $\cML\neq \emptyset$, all its components are noncompact, $\HH:=\partial\cML$ is a non-empty union of components of $\mu^{-1}(h)$ and hence compact, and $\cMR:=\overline{M\setminus\cML}$ contains $\mu^{-1}((-\infty,h])$ and is therefore noncompact. Remark \ref{rem:euler} and Proposition \ref{prop:euler} finish the proof.
\end{proof}

Cases (A) and (C) are the two configurations that a proper Morse function produces, and the number of ends distinguishes between them: a proper $\mu$ bounded below has regular levels bounding compact sublevel sets, while a proper $\mu$ unbounded in both directions has compact regular levels with two noncompact sides. No proper function can produce a noncompact $\HH$ at all, since $\mu^{-1}(h)$ would be the preimage of a compact set. This is the subject of the next case, in which properness must be abandoned. As we have mentioned in the introduction, this case is special in that the realization technique is necessarily ad hoc.

\begin{thmstar}[Case D]
Let $(M,g)$ be a connected open Riemannian manifold, $n\geq2$. Then $M$ admits an admissible collar-gluing in which $\cMR$, $\HH$ and $\cML$ are all noncompact. In fact, one can be found such that
\[
  \HH\cong\mathbb{R}^{n-1},\qquad
  \cMR\cong\mathbb{R}^{n-1}\times[0,\infty),\qquad
  \Rreg\cong\mathbb{R}^{n}.
\]
The choices of $M_R$ and $M_L$ can be interchanged.
\end{thmstar}

\begin{lem}\label{lem:ray}
Every connected open manifold $M$ admits a properly embedded ray $\gamma:[0,\infty)\hookrightarrow M$.
\end{lem}
\begin{proof}
Give $M$ a smooth triangulation~\cite{Whitehead}. Since $M$ is connected and noncompact, its $1$-skeleton is a connected, locally finite, infinite graph, so it contains an infinite simple edge path $v_{0},v_{1},\dots$ (this is due to K\"onig \cite{Konig}). Let $\gamma_{0}:[0,\infty)\to M$ traverse it. It is injective, and proper: a compact $C\subset M$ meets only finitely many simplices, so $\gamma_{0}^{-1}(C)$ is a closed subset of a finite union of edges. A proper injective continuous map into a Hausdorff space is a closed embedding. Finally, smooth the corners at the $v_{i}$ inside their open stars, which changes $\gamma_{0}$ inside a locally finite family of compact sets and so preserves properness and injectivity.
\end{proof}

\begin{proof}[Proof of Case (D)]
Let $\gamma$ be a properly embedded ray as in Lemma~\ref{lem:ray} and $R=\gamma([0,\infty))$, a closed properly embedded submanifold with boundary. Since $[0,\infty)$ is contractible, the normal bundle of $R$ is trivial, so $R$ has a closed tubular neighbourhood $T\cong[0,\infty)\times D^{n-1}$, closed as a subset of $M$. Rounding the corner along $\{0\}\times S^{n-2}$ (or capping it off with a half-ball) makes $T$ a smooth codimension-zero submanifold with smooth boundary
\[
  \HH:=\partial T\cong D^{n-1}\cup_{S^{n-2}}\bigl([0,\infty)\times S^{n-2}\bigr)
  \cong\mathbb{R}^{n-1},
\]
and $T$ itself, a half-infinite thickened ray with rounded corner, is diffeomorphic to the closed half-space $\mathbb{R}^{n-1}\times[0,\infty)$. See Figure \ref{fig:tube}.

Set $\cMR:=T$ and $\cML:=\overline{M\setminus T}$. Since $\partial T\neq\emptyset$ and $M$ is connected, $M\setminus T\neq\emptyset$, so $M\setminus\HH=\intr(T)\sqcup(M\setminus T)$. Thus have a collar-gluing. A component of $\cML$ with empty boundary would be open and closed in $M$, so $\cML$ is connected with $\partial\cML=\HH$. As $\HH\cong\mathbb{R}^{n-1}$ is noncompact and closed in $\cML$, the side $\cML$ is noncompact, and so is $\cMR=T$. Now apply Proposition \ref{prop:euler}. The same argument works with $\cML:=T$ and $\cMR:=\overline{M\setminus T}$.
\end{proof}

\begin{figure}[htbp]
  \centering
  \begin{tikzpicture}[line join=round,line cap=round,scale=0.72]
    \def\tubepath{(2.3,4.3) .. controls (4.0,4.7) and (4.6,1.9) .. (6.4,2.2)
                  .. controls (8.2,2.5) and (8.4,4.9) .. (10.2,4.6)
                  .. controls (12.0,4.3) and (12.4,3.5) .. (14.2,3.3)}
    \begin{scope}
      \clip (0,0.3) rectangle (13,6.15);
      \fill[lor] (0,0.3) rectangle (13,6.15);
      \draw[line width=0.91cm, black] \tubepath;
      \draw[line width=0.82cm, riem] \tubepath;
      \draw[densely dotted, black!65, line width=0.8pt] \tubepath;
    \end{scope}
    \fill (2.3,4.3) circle (2.2pt);
    \draw[black!45, dashed, line width=0.8pt] (0,0.3) rectangle (13,6.15);
    \draw[black!45,-{Stealth[length=5pt]},line width=0.8pt] (13.15,3.2) -- (13.85,3.2);
    \node at (0.85,5.6) {$M$};
    \node at (7.6,5.5) {$\Lreg$};
    \node at (4.2,1.1) {$\cMR$};
    \draw[black!70,-{Stealth[length=4pt]},line width=0.6pt] (4.35,1.5) -- (5.1,2.85);
    \node at (9.9,1.6) {$\HH$};
    \draw[black!70,-{Stealth[length=4pt]},line width=0.6pt] (9.9,1.95) -- (9.35,3.75);
    \node at (1.3,3.35) {$\gamma$};
  \end{tikzpicture}
  \caption{Case (D) inside a piece of an open manifold $M$. No part of the configuration is compact, and the two sides may be interchanged.}
  \label{fig:tube}
\end{figure}
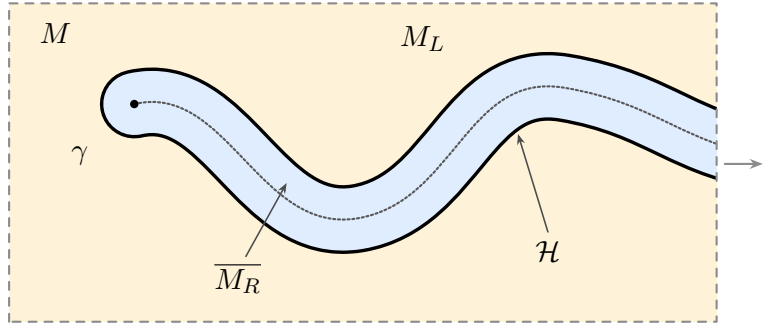

Near $\HH$ the generic function $f$ of Lemma \ref{lem:defining} (which went into our construction through Proposition \ref{prop:euler}) is a coordinate projection, which is Morse but not proper, and away from $\HH$ it is constant. Consequently no global Morse-theoretic input is used in case (D). Replacing the ray by other properly embedded objects gives further hypersurfaces: if $C\subset M$ is a properly embedded closed noncompact submanifold of dimension $k\leq n-1$ with trivial normal bundle, then $\partial T(C)$ is again a noncompact separating hypersurface, and for $C\cong\mathbb{R}^{k}$ one gets $\HH\cong\mathbb{R}^{k}\times S^{n-k-1}$. Taking $C$ a properly embedded noncompact graph when $n=3$, one realizes transition hypersurfaces of arbitrary genus and number of ends (as permitted by $M$).

\begin{corstar}[Case E]
\label{cor:E}
\label{Corollary:Geometric Continuation Case 4}
Let \((M,g)\) be a connected open Riemannian manifold with \(\dim M=n\geq2\). A collar-gluing presentation $M=\cMR\cup_{\HH}\cML$ with $\cML$ compact is admissible if and only if the criterion of Case (B) (Section \ref{subsec:cMR compact}) holds for the components of $\cML$.
\end{corstar}

\begin{proof}
The first statement follows from Proposition~\ref{prop:euler}, as in
Case~(B). It remains only to note the compactness properties of the two
sides. Since \(\overline{M_L}\) is compact and
\(\mathcal H\subset\overline{M_L}\) is closed in  \(\overline{M_L}\),
\(\mathcal H\) is compact. Moreover, since \(M\) is noncompact while
\(\overline{M_L}\) is compact, $\overline{M_R}=M\setminus M_L$
is noncompact.
\end{proof}


This completes the proof of the main theorem. That the list is exhaustive is immediate, as noted in Remark \ref{rem:selection rules exhaustive}.

\section{Remarks and examples}\label{sec:remarks}

\subsection{The sign of \texorpdfstring{$V$}{V}}

Let \(X\) be a connected component of \(M_L\). For an admissible continuation, \(V|_X\) is timelike and nowhere vanishing, and therefore determines a time orientation on \(X\). We take \(V\) to be future-pointing; thus a causal vector \(Y\) is future-pointing when $\widetilde g(Y,V)<0$. Replacing \(V\) by \(-V\) leaves \(\widetilde g\) unchanged, since \(V\) enters~\eqref{eq:ansatz} quadratically, but reverses the chosen time orientation.

Let \(\mathcal H_j\subset\partial\overline X\) be a connected boundary component. If \(V\) points from \(\mathcal H_j\) into \(X\), then \(\mathcal H_j\) is an initial hypersurface, or big bang; if \(V\) points from \(X\) across \(\mathcal H_j\) into \(M_R\), it is a final hypersurface, or big crunch. These labels are assigned relative to each connected component \(X\) of \(M_L\).

If \(X\) has several boundary components, there is no requirement that exactly one be initial and one final: several components may be initial, or several may be final, provided the relative Euler-class condition is satisfied. This does not require \(V\) to vanish or the time orientation to reverse; it only means that the direction of \(V\) relative to the different boundary components is different. Distinct connected components of \(M_L\) are treated independently.
Thus Proposition~\ref{prop:euler}(4) requires one to assign on a compact Lorentzian region bangs and crunches on its boundary $\HH$ in such a way that the Euler condition is met (see Example \ref{ex:caps} and Figure \ref{fig:caps}). In even dimensions, and in particular for $n=4$, no assignment is admissible if $\chi(\cML)\neq0$.

\subsection{A Morse-theoretic causal selection rule}

The mere existence (\cite[Proposition 37]{ONeill}) of a Lorentzian metric is insufficient for a physically satisfactory spacetime. It neither guarantees a global time function nor excludes branching of the global time evolution. This distinction becomes apparent in examples involving topology change, such as the trouser topology, where a single connected ``tube'' splits into two; Figure \ref{fig:trouser}. In this simplest example, the number of connected components of the
spatial sections changes from one to two, so that their \(\pi_0\)
changes as one passes from the Riemannian region through the
topology-changing regime into the Lorentzian region.

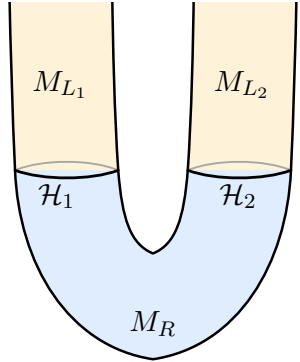
\begin{figure}[htbp]
\centering
\begin{tikzpicture}[line join=round,line cap=round,scale=0.5]
  \def\bodypath{(1.00,10.35)
    .. controls (1.03,8.80) and (1.08,7.10) .. (1.15,5.58)
    .. controls (1.25,3.10) and (2.55,0.92) .. (4.80,0.58)
    .. controls (7.05,0.92) and (8.34,3.10) .. (8.45,5.58)
    .. controls (8.52,7.10) and (8.57,8.80) .. (8.60,10.35) -- cycle}

  \def\holepath{(3.72,10.38)
    .. controls (3.74,8.82) and (3.78,7.05) .. (3.88,5.58)
    .. controls (3.98,4.30) and (4.20,3.62) .. (4.80,3.38)
    .. controls (5.40,3.62) and (5.62,4.30) .. (5.72,5.58)
    .. controls (5.82,7.05) and (5.86,8.82) .. (5.88,10.38) -- cycle}

  \begin{scope}
    \clip (0.50,0.25) rectangle (9.10,10.02);
    \begin{scope}[even odd rule]
      \clip \bodypath \holepath;
      \fill[riem] (0.50,0.25) rectangle (9.10,10.02);
      \fill[lor]  (0.50,5.58) rectangle (9.10,10.02);
    \end{scope}

    \draw[line width=0.95pt] (1.00,10.35)
      .. controls (1.03,8.80) and (1.08,7.10) .. (1.15,5.58)
      .. controls (1.25,3.10) and (2.55,0.92) .. (4.80,0.58)
      .. controls (7.05,0.92) and (8.34,3.10) .. (8.45,5.58)
      .. controls (8.52,7.10) and (8.57,8.80) .. (8.60,10.35);

    \draw[line width=0.95pt] (3.72,10.38)
      .. controls (3.74,8.82) and (3.78,7.05) .. (3.88,5.58)
      .. controls (3.98,4.30) and (4.20,3.62) .. (4.80,3.38)
      .. controls (5.40,3.62) and (5.62,4.30) .. (5.72,5.58)
      .. controls (5.82,7.05) and (5.86,8.82) .. (5.88,10.38);

    \draw[black!35,line width=0.70pt]
      (1.15,5.58) .. controls (1.88,5.88) and (3.16,5.88) .. (3.88,5.58);
    \draw[line width=1.05pt]
      (1.15,5.58) .. controls (1.88,5.31) and (3.16,5.31) .. (3.88,5.58);

    \draw[black!35,line width=0.70pt]
      (5.72,5.58) .. controls (6.44,5.88) and (7.72,5.88) .. (8.45,5.58);
    \draw[line width=1.05pt]
      (5.72,5.58) .. controls (6.44,5.31) and (7.72,5.31) .. (8.45,5.58);
  \end{scope}

  \node at (2.35,7.85) {$M_{L_1}$};
  \node at (7.15,7.85) {$M_{L_2}$};
  \node at (4.80,1.52) {$M_R$};
  \node at (2.25,4.90) {$\mathcal H_1$};
  \node at (7.05,4.90) {$\mathcal H_2$};
\end{tikzpicture}
\caption{The noncompact trouser manifold. The bifurcation point corresponds to a critical point of the (height) Morse function. Choosing the signature-changing hypersurface \(\mathcal H\) above the critical level places the topology-changing event on the Riemannian side, while the Lorentzian region splits into two disconnected cylindrical components.}
\label{fig:trouser}
\end{figure}

Classical results of Geroch~\cite{Geroch} and subsequent authors~\cite{Borde,Dowker+Surya,Dowker+Garcia+Surya,Garcia-Heveling} show that topology change is severely restricted in Lorentzian spacetimes satisfying suitable causal assumptions, although the precise conclusions depend on the compactness assumptions and on the global structure of the spacelike hypersurfaces involved. Thus, while manifolds exhibiting topology change may admit Lorentzian metrics, such metrics need not possess desirable global causal properties.

The following classical result due to Geroch \cite{Geroch} illustrates this:

\begin{quote}
\emph{The chronology condition is violated on nonempty closed Lorentzian manifolds.}\footnote{Recall that a time-oriented Lorentzian manifold is said to satisfy the \emph{chronology condition} if it admits no closed timelike curves \cite[p.\ 407]{ONeill}.}
\end{quote}

This suggests the need for a Morse-theoretic selection rule, or ``choice of hypersurface'' admissibility condition. 

It is a standard fact that two level sets of a Morse function between which there are no critical points are diffeomorphic. We can strengthen that statement slightly by assuming properness:

\begin{cor}[Morse-causal admissibility of the Lorentzian region]
\label{cor:selection}
	Suppose $M$ is the interior of a compact smooth manifold $W$ with boundary. Then $M$ admits an admissible collar-gluing $M=\cMR\cup_{\HH}\cML$ with $\cMR$ compact and with a diffeomorphism $\cML\cong\HH\times[0,\infty)$ of smooth manifolds with boundary.
\end{cor}

\begin{proof}
	We will first show that $M$ admits a proper Morse function $\mu$ with a regular value $h$ such that $\cML=\mu^{-1}[h,\infty)\cong\HH\times[0,\infty)$, where $\HH=\mu^{-1}h$.
	
	It is standard (see \cite[\S 2]{Milnor65}) that $W$ admits a Morse function $f$ (meaning it is Morse on $M=\intr W$) with $\partial W=f^{-1}0$ and without critical points in a neighbourhood near $\partial W$. Its domain being compact, $f$ is proper and has finitely many critical points. Taking $\mu=1/f|_{M}$ we obtain a proper Morse function on $M$.
	
	Now we proceed as in the proof of the fact that for $h$ a regular value of $\mu$ above all its critical values we have diffeomorphisms $\mu^{-1}h\cong\mu^{-1}h'$ for all $h'\geq h$ \cite[Theorem 3.1]{Milnor1963}. On $\mu^{-1}([h,\infty))$ the field $Y=\nabla\mu/|\nabla\mu|^{2}$ is defined and satisfies $d\mu(Y)=1$. Since $\mu$ is proper, $\mu^{-1}([h,b])$ is compact for each $b\geq h$, and so an integral curve of $Y$ starting on $\HH$ stays in $\mu^{-1}([h,b])$ up to time $b-h$. The flow $\varphi$ of $Y$ on $\mu^{-1}([h,\infty))$ is defined for all time $s\geq0$ (and backwards until the trajectory meets $\HH$). This is the vector field used in the proof of \cite[Theorem 3.1]{Milnor1963} along whose flow $\mu$ increases at unit speed. Hence $\HH\times[0,\infty)\to\mu^{-1}([h,\infty))$, $(x,s)\mapsto\varphi_{s}(x)$, is a diffeomorphism, with inverse $q\mapsto(\varphi_{h-\mu(q)}(q),\mu(q)-h)$.
	
	Now admissibility follows from Case (A) (Section \ref{subsec:cMR compact}).
\end{proof}

\begin{rem}
	The hypothesis of Corollary \ref{cor:selection}, that an open smooth manifold $M$ be the interior of a compact smooth manifold $W$ with boundary, is equivalent to it admitting a proper Morse function with finitely many critical points. We proved one implication in the proof itself; for the other implication, a flow construction like in that proof gives a diffeomorphism $M\cong \intr f^{-1}[a,b]$ for suitable $a,b$ beyond the critical values, and one can take $W=f^{-1}[a,b]$. The problem of finding a boundary is a classical problem in geometric topology (for $\dim\geq6$ it is classical: see Siebenmann \cite{Siebenmann}).
	
	Properness cannot be dropped for a conceptual reason: as a direct consequence of Phillips' submersion theorem, every open smooth manifold admits a submersion $M\to \mathbb R$ as it admits a nowhere-vanishing vector field \cite[Corollary 8.3]{Phillips}.
\end{rem}

This corollary should not be read as a direct application of (or answer to) the result of Geroch we recalled above, whose standard form applies to the closed case. It is only motivated by it, noting an intrinsic condition on $M$ (having a smooth boundary at infinity) which lets us maximally restrict the topology of the Lorentzian region given $\HH$.

\subsection{Examples}

\begin{example}[caps on the sphere]
\label{ex:caps}
Let $M=S^{n}$ and let $\HH=\HH_{1}\sqcup\HH_{2}$ be two disjoint round spheres (each an $S^{n-1}$) bounding the polar caps, and let $\cMR$ be the union of the two caps. Thus $\cML\cong S^{n-1}\times[0,1]$; see Figure \ref{fig:caps}. This is the closed-universe version of the no-boundary picture: a Lorentzian era of spherical spatial sections, capped at both ends by Euclidean regions instead of being terminated by singularities. Here $\chi(\cML)=\chi(S^{n-1})$, which is $0$ for $n$ even and $2$ for $n$ odd, so:
\begin{itemize}
  \item For $n$ even, and in particular for $n=4$, the configuration is admissible, and every distribution of big bangs and big crunches over $\HH_{1},\HH_{2}$ can be realised.
  \item For $n$ odd, a distribution is admissible if and only if exactly one of $\HH_{1},\HH_{2}$ is a big bang and the other a big crunch. The other distributions are obstructed by $\chi(\cML)-\chi(\partial_{-}\cML)=\pm2$.
\end{itemize}
The second statement is a selection rule in the strict sense: the topology of the Lorentzian region determines how many beginnings and endings it may have. It also illustrates that the criterion $\chi(\cML)=0$ is too strong in odd dimensions.

In particular, when they are admissible, the two-big-bang and two-big-crunch configurations do not represent a reversal of time orientation: \(V\) remains a smooth nowhere-vanishing timelike field throughout \(M_L\). What fails in such a configuration is not time-orientability, but the existence of a temporal function whose level sets run monotonically from one transition component to the other.
\end{example}

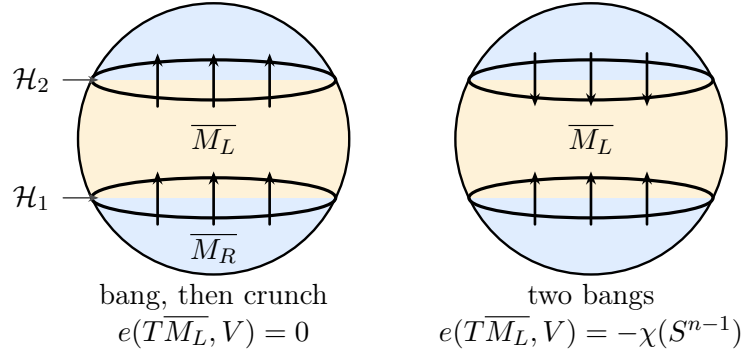
\begin{figure}[htbp]
  \centering
  \begin{tikzpicture}[line join=round,line cap=round,scale=0.78]
    \def\lo{2.0}
    \def\hi{4.0}
    \newcommand{\capsphere}[1]{%
      \begin{scope}[shift={(#1,3.0)}]
        \begin{scope}
          \clip (0,0) circle (2.3);
          \fill[riem] (-2.5,-2.5) rectangle (2.5,2.5);
          \fill[lor]  (-2.5,-1.0) rectangle (2.5,1.0);
        \end{scope}
        \draw[line width=0.9pt] (0,0) circle (2.3);
        \draw[line width=1.2pt] (0,-1.0) ellipse (2.07 and 0.34);
        \draw[line width=1.2pt] (0,1.0)  ellipse (2.07 and 0.34);
      \end{scope}}
    \capsphere{2.8}
    \foreach \x in {-0.95,0,0.95}
      {\draw[-{Stealth[length=5pt]},line width=1.0pt] (2.8+\x,\lo-0.45) -- (2.8+\x,\lo+0.45);
       \draw[-{Stealth[length=5pt]},line width=1.0pt] (2.8+\x,\hi-0.45) -- (2.8+\x,\hi+0.45);}
    \node at (2.8,3.0) {$\cML$};
    \node at (2.8,1.15) {$\cMR$};
    \node[anchor=east] at (0.2,\lo) {$\HH_{1}$};
    \draw[black!70,-{Stealth[length=4pt]},line width=0.6pt] (0.25,\lo) -- (0.85,\lo);
    \node[anchor=east] at (0.2,\hi) {$\HH_{2}$};
    \draw[black!70,-{Stealth[length=4pt]},line width=0.6pt] (0.25,\hi) -- (0.85,\hi);
    \node[align=center] at (2.8,0.0) {bang, then crunch\\[1pt] $e(T\cML,V)=0$};
    \capsphere{9.2}
    \foreach \x in {-0.95,0,0.95}
      {\draw[-{Stealth[length=5pt]},line width=1.0pt] (9.2+\x,\lo-0.45) -- (9.2+\x,\lo+0.45);
       \draw[-{Stealth[length=5pt]},line width=1.0pt] (9.2+\x,\hi+0.45) -- (9.2+\x,\hi-0.45);}
    \node at (9.2,3.0) {$\cML$};
    \node[align=center] at (9.2,0.0) {two bangs\\[1pt] $e(T\cML,V)=-\chi(S^{n-1})$};
  \end{tikzpicture}
  \caption{The two distributions of Example~\ref{ex:caps}, with $M=S^{n}$ cut
along two spheres and $\cML\cong S^{n-1}\times[0,1]$ the band between the caps. 
}
  \label{fig:caps}
\end{figure}

\begin{example}[Compact Lorentzian sides and pseudo-causal return]
Let \(M\) be a torus decomposed by two disjoint signature-changing
hypersurfaces $\mathcal H_1,\mathcal H_2\subset M$ into a Riemannian region \(M_R\) and a Lorentzian region \(M_L\), so that
both closed sides
\[
    \overline{M_R}=M_R\cup\mathcal H_1\cup\mathcal H_2,
    \qquad
    \overline{M_L}=M_L\cup\mathcal H_1\cup\mathcal H_2
\]
are compact. The torus example illustrates that the compact--compact case is not the same as a compact Lorentzian spacetime without boundary. If the whole torus carried a nondegenerate Lorentzian metric, then compactness would force the existence of a closed timelike curve. This classical conclusion does not apply directly here, because the Lorentzian metric is defined only on the open region \(M_L\) and becomes degenerate on the boundary hypersurfaces \(\mathcal H_1\) and \(\mathcal H_2\). Suppose, moreover, that the Lorentzian side admits a temporal function
\[
    \tau:\overline{M_L}\to[0,1]
\]
with
\[
    \tau^{-1}(0)=\mathcal H_1,
    \qquad
    \tau^{-1}(1)=\mathcal H_2,
\]
such that \(d\tau\) is timelike on \(M_L\). We choose the time orientation so that \(\tau\) increases toward the future. This is an additional causal assumption beyond time-orientability; in particular, it forces \(\mathcal H_1\) to be initial and \(\mathcal H_2\) to be final, and therefore excludes the two-big-bang and two-big-crunch assignments allowed in Example~\ref{ex:caps}. Then \(\tau\) is strictly
monotone along every future-directed timelike curve in \(M_L\). Hence
\(M_L\) contains no closed future-directed timelike curve. If \(\tau\) is a Cauchy temporal function, then every inextendible future-directed causal curve in \(M_L\) runs from the initial transition hypersurface \(\mathcal H_1\) toward the final transition hypersurface \(\mathcal H_2\).

Thus, in this signature-changing model, the compactness of \(\overline{M_L}\) does not create the closed-timelike-curve pathology. The degenerate hypersurfaces act as the initial and final boundaries of the Lorentzian era. If one extends the discussion to pseudo-timelike curves on the full signature-changing manifold then a curve may leave \(M_L\) through one signature-changing hypersurface, pass through the Riemannian region, and re-enter \(M_L\) through the other. Such a curve is not a closed timelike curve, but a pseudo-causal return, in the sense of the pseudo-timelike curves of~\cite{Hasse + Rieger-Loops}.


\end{example}

\section{Outlook}

The cases considered in this article give geometro-topological selection rules for admissible Riemannian-Lorentzian continuations of the form~\eqref{eq:ansatz}.  They do not determine which of the resulting signature-changing metrics are instantons. That requires further analytic and physical conditions, including the Einstein equations, matter-field equations, regularity conditions at the transition hypersurface, and suitable action and boundary conditions. Thus the selection rules obtained here should be regarded as necessary geometro-topological constraints on gravitational or cosmological instanton candidates, rather than as sufficient criteria for an instanton.

Our results concern an earlier stage of the problem: given a Riemannian region, when can it be joined to a Lorentzian region across a transverse signature-changing hypersurface? The resulting selection rules apply therefore to instanton candidates for which such a Lorentzian continuation is sought. This includes cosmological instantons, as well as gravitational instantons which are expected
to come from, or to continue to, Lorentzian solutions.

This distinction is important because gravitational instantons need not, in general, admit Lorentzian continuations (see \cite{Dunajski - Gravitational Instantons Old and New}): Some examples, such as Euclidean Schwarzschild and Euclidean Kerr, arise by analytic continuation of Lorentzian black-hole metrics. Other examples, such as Eguchi--Hanson, anti-self-dual Taub--NUT, and Chen--Teo-type instantons, do not arise in this way.

The compact-Riemannian/noncompact-Lorentzian configuration includes the standard no-boundary picture, in which a compact Riemannian cap is continued across a compact hypersurface to an expanding Lorentzian spacetime. It also includes Coleman--De Luccia type continuations; in that setting, the Euclidean saddle may be topologically similar to the no-boundary cap, but it carries different metric and matter data, and the Lorentzian continuation contains an expanding bubble separating regions with different effective vacuum energy. The selection rules developed here do not distinguish Hartle--Hawking and Coleman--De Luccia models, since both belong to the same compactness class. 

\section*{Acknowledgements}NER thanks Matthew Kleban for many insightful discussions that helped shape the direction of this article. This research was partially supported by the Austrian Science Fund (FWF) [Grant DOI 10.55776/EFP6]. \"{O}T was supported by the Austrian Science Fund (FWF) through FWF Project no.\ P 37046. We thank the Yale Mathematics Department and the Institute for Mathematics of the University of Zurich for their hospitality.

\setlength{\parskip}{0pt}
\end{document}